\documentclass[aos]{imsart}

\RequirePackage{amsthm,amsmath,amsfonts,amssymb}
\RequirePackage[numbers]{natbib}
\RequirePackage[colorlinks,citecolor=blue,urlcolor=blue]{hyperref}
\RequirePackage{graphicx}
\RequirePackage{xcolor}

\usepackage{xr}
\startlocaldefs

\newtheorem{theorem}{Theorem}[section]
\newtheorem*{theoConv}{Theorem \ref{thm_contraction_tests}}
\newtheorem{proposition}[theorem]{Proposition}
\newtheorem{lemma}[theorem]{Lemma}
\newtheorem*{lemmaConj}{Lemma \ref{Lem:Conjugacy}}
\newtheorem*{lemmaSmallBall}{Lemma \ref{Lemma:AnotherSmallBall}}
\newtheorem*{lemmaPriorProb}{Lemma \ref{lem:prior_prob_event}}
\newtheorem*{lemmaBias}{Lemma \ref{lem:bias}}
\newtheorem*{CorollaryPostMean1}{Corollary \ref{Cor:PostMean1}}
\newtheorem*{CorollaryPostMean2}{Corollary \ref{Cor:PostMean2}}
\newtheorem{corollary}[theorem]{Corollary}
\theoremstyle{remark}
\newtheorem{remark}{Remark}[section]

\newtheorem{example}{Example}[section]
\newtheorem{condition}{Condition}[section]

\renewcommand{\H}{\mathbb{H}}

\newcommand{\R}{\mathbb{R}}
\newcommand{\T}{\mathbb{T}}

\newcommand{\Z}{\mathbb{Z}}
\newcommand{\N}{\mathbb{N}}
\newcommand{\mF}{\mathcal{F}}

\newcommand{\mB}{\mathcal{B}}

\DeclareMathOperator*{\argmin}{arg\,min}
\newcommand{\eps}{\varepsilon}

\newenvironment{enumerate*}%

\makeatletter 
\@ifundefined{@currsizeindex} 
  {\RequirePackage{relsize}\let\dolarger\relsize} 
  {\def\dolarger#1{\larger[#1]}} 

\newcommand*\@@bigtimes[2]{\vphantom{\prod} 
  \vcenter{\hbox{\dolarger{4}$\m@th#1\mkern-2mu\times\mkern-2mu$}}} 
\newcommand*\bigtimes{\mathop{\mathpalette\@@bigtimes\relax}\displaylimits} 
\makeatother

\def\B{\mathbb{B}}\def\G{\mathbb{G}}\def\H{\mathbb{H}}\def\N{\mathbb{N}}\def\R{\mathbb{R}}\def\Z{\mathbb{Z}}\def\1{\mathbbm{1}}

\def\Bcal{\mathcal{B}}\def\Fcal{\mathcal{F}}\def\Lcal{\mathcal{L}}\def\Qcal{\mathcal{Q}}\def\Scal{\mathcal{S}}\def\Zcal{\mathcal{Z}}

\endlocaldefs

\begin{document}

\begin{frontmatter}
\title{Nonparametric Bayesian inference for reversible multi-dimensional diffusions}
\runtitle{Bayesian reversible diffusions}

\begin{aug}
\author[A]{\fnms{Matteo} \snm{Giordano}\ead[label=e1]{mg846@cam.ac.uk}}
and
\author[B]{\fnms{Kolyan} \snm{Ray}\ead[label=e2]{kolyan.ray@imperial.ac.uk}}
\address[A]{Department of Pure Mathematics and Mathematical Statistics, University of Cambridge,
\printead{e1}}

\address[B]{Department of Mathematics,
Imperial College London,
\printead{e2}}
\end{aug}

\begin{abstract}
We study nonparametric Bayesian models for reversible multidimensional
diffusions with periodic drift. For continuous observation paths, reversibility is exploited to prove a general posterior contraction rate theorem for the drift gradient vector field under approximation-theoretic conditions on the induced prior for the invariant measure. The general theorem is applied to Gaussian priors and $p$-exponential priors, which are shown to converge to the truth at the optimal nonparametric rate over Sobolev smoothness classes in any dimension.
\end{abstract}

\begin{keyword}[class=MSC2020]
\kwd[Primary ]{62G20}
\kwd[; secondary ]{62F15}
\kwd{60J60}
\end{keyword}

\begin{keyword}
\kwd{Bayesian nonparametrics}
\kwd{multi-dimensional diffusions}
\kwd{reversibility}
\kwd{Gaussian processes}
\kwd{Laplace prior}
\end{keyword}

\end{frontmatter}
\tableofcontents

\section{Introduction}

Consider observing a continuous trajectory $X^T = (X_t= (X_t^1,\dots,$ $X_t^d): 0\leq t \leq T)$ of the multi-dimensional Markov diffusion process given by the solution to the stochastic differential equation (SDE)
\begin{equation}\label{model}
	dX_t = b(X_t) dt + dW_t, \qquad X_0 = x_0 \in \R^d, \quad t\geq 0,
\end{equation}
where $W_t = (W_t^1,\dots,W_t^d)$ is a standard Brownian motion on $\R^d$ and $b = (b_1,\dots,b_d)$ is a Lipschitz vector field. We are interested in nonparametric Bayesian inference on the drift term $b$ in the practically important case when the diffusion process is \textit{time-reversible}. By a result of Kolmogorov, this is equivalent to the drift $b$ equalling the gradient vector field $\nabla B$ of a potential function $B: \R^d \to \R$ (e.g.~\cite{bakry2014}, p. 46).

The SDE \eqref{model} with $b = \nabla B$ is the Brownian dynamics model in physics, describing the trajectory of a particle diffusing in a potential energy field that exerts a force directed towards its local extrema. One is often interested in inference on the potential $B$ (or drift $\nabla B$), which carries important physical information, based on an observed trajectory of the particle. For example, such dynamics arise in the Smoluchowski-Kramers approximation to the Langevin equation for the motion of a chemically bound particle \cite{K40,S80}, in which case $B$ describes the chemical bonding forces. Other applications of the reversible diffusion model \eqref{model} in physics and chemistry include vacancy diffusion and Lennard-Jones clusters \cite{PS10,pinski2012} and chemical reaction equations \cite{bolhuis2002}. Once fitted, such diffusive models can also be employed as computationally affordable emulators of the physical process \cite{GSMH19}.

	A Bayesian who wants to model reversible diffusion dynamics must do so explicitly via the prior, namely by constructing one that draws gradient vector fields for $b$. The natural approach is then to directly place a prior on the potential $B$ rather than on $b$, which is the approach we pursue here. Moreover, $B$ typically has a strong physical meaning and estimating it is often the primary inferential goal, in which case explicitly modelling the potential provides interpretable inference. Our aim is to provide theoretical convergence guarantees as the time horizon $T\to\infty$ for this Bayesian approach, which arises directly from physical considerations in the modelling step.

	The theoretical performance of nonparametric Bayesian procedures for drift estimation is well-studied in the one-dimensional case $d=1$ \cite{vdMeulen2006,PSvZ13,GS14,vWvZ16,NS17,A18,vW2019}. However, much less is known in the general multi-dimensional setting $d\geq 2$. In the continuous observation model, Dalalyan and Rei\ss ~\cite{DR07} established pointwise convergence rates and Strauch \cite{S15,S16,S18} obtained adaptive rates, both using multivariate kernel-type estimators. Schmisser \cite{S13} established adaptive $L^2$ convergence rates for certain penalised least-squares estimators with discrete observations. Nickl and Ray \cite{nicklray2020} obtained $L^2$ and $L^\infty$ posterior contraction rates, as well as Bernstein-von Mises results, for certain \textit{non-reversible} drift vector fields - more discussion can be found below.

	In this paper, we obtain contraction rates for the posterior distribution of $B$ about the true potential $B_0$ in the model \eqref{model} as $T\to\infty$. We prove a general theorem for diffusions governed by gradient vector fields based on the classical testing approach of Bayesian nonparametrics \cite{ghosal2000,vdvaart2008}, and firstly apply it to Gaussian priors for $B$, obtaining optimal rates over Sobolev smoothness classes in any dimension. Gaussian priors are widely used in diffusion models \cite{PSvZ13,RBO13,GS14,vWvZ16,vW2019,BRO18,nicklray2020} and are a canonical choice, not least for computational reasons \cite{PPRS12, RBO13,vdMS17}. Our results thus provide statistical convergence guarantees for these widely used priors. As a consequence, we also deduce convergence rates for the corresponding maximum a-posteriori (MAP) estimates, which can be viewed as penalized maximum likelihood estimators. Apart from the results in \cite{nicklray2020}, these are the first multi-dimensional Bayesian contraction results for diffusions, and the first for potential-modelling priors.

	In applications, the potential $B$ is often spatially inhomogeneous for physical reasons, for instance being spiky in some regions and flat or smooth in others (e.g.~molecular communication \cite{GLCAP11}), see Figure \ref{fig:het_potential} for an illustrative example. Gaussian priors are known to be unsuited to modelling such inhomogeneities (see e.g.~\cite{ADH20,AW21,GRSH22})), which has motivated the use of heavier tailed priors, especially Besov space priors in the inverse problems \cite{R13,ABDH18,DS17,ADH20,AW21} and medical imaging \cite{SE15,VEtAl09} communities. To address such physically motivated situations, we also consider modelling the potential using heavier-tailed $p$-exponential priors \cite{ADH20}, for which we establish optimal nonparametric rates. Such priors have been employed to recover spatially inhomogeneous functions, such as those with a `blocky' structure with sudden changes from one block to another. They  have a number of attractive properties including edge preservation, discretization invariance \cite{LSS09,L12}, promoting sparse solution representations, while also maintaining a log-concave structure that aids posterior sampling. Contraction rates for such priors have recently been obtained in direct linear models \cite{ADH20}, and we present here a first extension to a nonlinear diffusion setting.

\begin{figure}\label{fig:het_potential}
\includegraphics[width=0.495\textwidth]{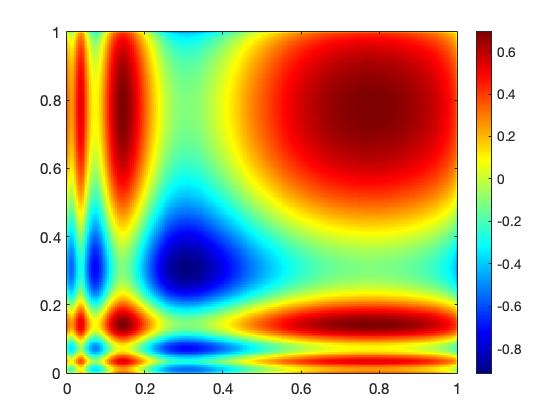}
\includegraphics[width=0.495\textwidth]{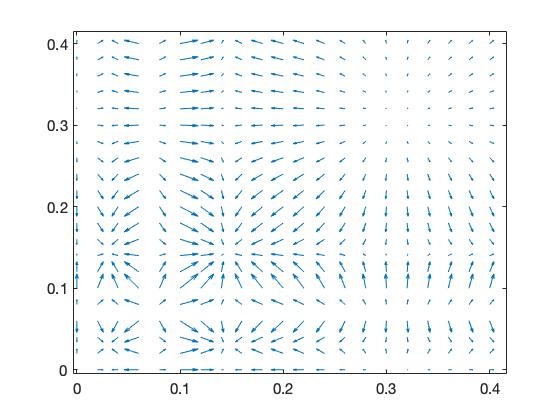}
\caption{An example of a spatially heterogeneous potential $B$ (left) and the corresponding gradient vector field $\nabla B$ (right). Note the axis scales are different for clarity.}
\end{figure}

	The multi-dimensional case $d \geq 2$ is intrinsically more challenging than the one-dimensional case. The testing approach has been used for model \eqref{model} first by van der Meulen et al. \cite{vdMeulen2006} to obtain contraction rates in the `natural distance' induced by the statistical experiment, which is a `random Hellinger semimetric' depending on the observation path $(X_t: 0 \leq t \leq T)$. In dimension $d=1$, the theory of diffusion local times can then be used to relate this random path-dependent metric to the $L^2$-distance \cite{vdMeulen2006, PSvZ13, vWvZ16}, but when $d>1$ such local time arguments are not available. In the multi-dimensional setting, Nickl and Ray \cite{nicklray2020} relate this random Hellinger metric to the $L^2$-distance for specific truncated Gaussian series product priors on $b$. They exploit concentration properties of the high-dimensional random matrices induced by the Hellinger semimetric on finite-dimensional projection spaces to relate this problem to a random design type regression problem. However, this approach crucially uses that the priors for each coordinate of $b = (b_1,\dots,b_d)$ are supported on the same finite-dimensional projection spaces, which is typically only the case if $b_1,\dots,b_d$ have \textit{independent} priors. Since product priors for $b$ draw gradient vector fields $b= \nabla B$ with probability zero, they are inherently unable to model reversibility.

	Aside from requiring different priors, modelling the potential $B$ introduces fundamentally new features to the inference problem at hand. Whereas one can relate the non-gradient vector field case to a direct linear regression type model \cite{nicklray2020}, modelling $B$ is equivalent to modelling the invariant measure (see \eqref{Eq:InvMeas}) and leads to a genuinely nonlinear regression problem. We must thus employ a completely different approach to \cite{nicklray2020} here, using instead tools from the statistical theory of nonlinear inverse problems  \cite{MNP20,AN19,GN20,NS20,MNP20b} - our work can thus also be viewed as a contribution to the Bayesian inverse problems literature.

	Instead of Hellinger testing theory, we develop concentration inequalities for preliminary estimators to directly construct suitable plug-in tests following ideas in \cite{GN2011}, see also \cite{R13,NS17,A18,AN19,MRS20}. In the present setting, the invariant measure $\mu_b$ of the diffusion describes the probabilities 
\begin{equation} \label{erginv}
\mu_b(A) =^{a.s.} \lim_{T \to \infty} \frac{1}{T}\int_0^T 1_A(X_t)dt,
\end{equation}
corresponding to the average asymptotic time spent by the process $(X_t)$ in a given measurable subset $A$ of the state space. In the reversible case, we can exploit the one-to-one correspondence $\mu_b \propto e^{2B}$ between the potential $B$ and the invariant measure $\mu_b$, to construct estimators for $B$ based on estimators for $\mu_b$. We then combine elliptic PDE and martingale techniques with concentration of measure arguments to obtain exponential inequalities for such estimators, and hence bounds for the type-II errors of suitable tests.

	We study posterior asymptotics in the ergodic setting, where one can gain information about the average behaviour of the diffusion dynamics from a sufficiently long particle trajectory. This requires strong enough recurrence that the diffusion is largely confined to a bounded region, which matches the physical intuition in many applied situations. To ensure such recurrence and mixing of the diffusion, we follow \cite{PPRS12,PSvZ13,vWvZ16,A18,GSMH19,vW2019,nicklray2020} in restricting to the \textit{periodic} setting, and thus periodic potentials $B$. Under this simplification, a potential $B$ still implies the corresponding (periodised) diffusion is reversible (\cite{GSMH19}, Proposition 2) and so our results maintain the key modelling link between reversibility and potential functions. Periodicity simplifies certain technical arguments, in particular the elliptic PDE techniques involved in studying the mapping properties of the generator of the underlying semigroup (see e.g. Chapter II.3 in \cite{BJS64}), thus allowing a cleaner exposition of the main statistical ideas. The underlying PDE techniques extend conceptually to the non-periodic setting under certain conditions, albeit at the expense of significant additional technicalities that are beyond the scope of the present paper, see Section \ref{sec:nonperiod} for discussion.

	As well as deriving theoretical results, we also discuss numerical implementation in the continuous observation model considered here. Posterior sampling using simulation techniques is well-studied for `real-world' discrete data, including for some priors we consider here, which we highlight when discussing concrete prior choices. For recent work on this active research topic, see for example \cite{BPRF06,PPRS12, RBO13,vdMS17, vdMS17b,BRO18,GSMH19} and references therein. 

\section{Main results}
\label{Sec:results}

\subsection{Basic notation and definitions}
\label{Sec:Basics}

Let $\T^d$ be the $d$-dimensional torus, isomorphic to $[0,1]^d$ with the opposite points on the cube identified.  We denote by $L^p(\T^d)$ the usual Lebesgue spaces on $\T^d$ equipped with norm $\|\cdot\|_p$, and by $\langle\cdot,\cdot\rangle_2$ the inner product on $L^2(\T^d)$. We further define the subspaces
$$\dot{L}^2(\T^d) := \left\{ f\in L^2(\T^d): \int_{\T^d}f dx = 0 \right\}, \qquad \dot{L}^2_\mu(\T^d) := \left\{f\in L^2(\T^d) : \int_{\T^d}fd\mu = 0 \right\},$$
where $\mu$ is a probability measure on $\T^d$.

	Let $C(\T^d)$ be the space of continuous functions on $\T^d$, equipped with the supremum norm $\|\cdot\|_\infty$. For $s>0$, denote by $C^s(\T^d)$ the usual H\"older space of $\lfloor s \rfloor$-times continuously differentiable functions on $\T^d$ whose $\lfloor s \rfloor^{\text{th}}$-derivative is $(s-\lfloor s\rfloor)$-H\"older continuous. We let $H^s(\T^d)$, $s\in \R$, denote the usual $L^2$-Sobolev spaces on $\T^d$, defined by duality when $s<0$. We further define the Sobolev norms $\|f\|_{W^{1,q}} = \|f\|_q + \sum_{i=1}^d \|\partial_{x_i} f\|_q$, and note that $\|\cdot\|_{W^{1,2}}$ is equivalent to $\|\cdot\|_{H^1}$.

	Let $\{\Phi_{l r}: ~ l \in \{-1,0\} \cup \N, ~r=0,\dots,\max(2^{l d}-1,0)\}$
be an orthonormal tensor product wavelet basis of $L^2(\T^d)$, obtained from a periodised Daubechies wavelet basis of $L^2(\T)$, which we take to be $S$-regular for $S\in \N$ large enough; see Section 4.3. in \cite{ginenickl2016} for details. For $J\in \N$, define the finite-dimensional approximation space
\begin{equation}\label{V_J}
	V_J:=\textnormal{span}\{\Phi_{l r}:~ l \le J,~r=0,\dots,\max(2^{l  d-1}-1,0)\}
\end{equation}
and let $P_J:L^2(\T^d)\to V_J$ be the associated $L^2$-projection operator. Note that $V_J$ has dimension $v_J:=\textnormal{dim}(V_J)=O(2^{Jd})$ as $J\to\infty$. For $0 \leq t \leq S$, $1\le p,q\le\infty$, define the Besov spaces via their wavelet characterisation:
$$
	B^t_{p q}(\T^d)=\left\{ f\in L^p(\T^d): \ \|f\|_{B^t_{p q}}^q
	:=\sum_{l}2^{q l \left(t+\frac{d}{2}-\frac{d}{p}\right)}
	\left(\sum_r|\langle f,\Phi_{l r}\rangle_2 |
	^p\right)^\frac{q}{p}<\infty \right\},
$$
replacing the $\ell_p$ or $\ell_q$-norm above with $\ell_\infty$ if $p=\infty$ or $q=\infty$, respectively. Recall that $H^t(\T^d)=B^t_{22}(\T^d)$ and the continuous embedding  $C^s(\T^d)\subseteq B^s_{\infty\infty}(\T^d)$ for $s\ge0$, see Chapter 3 in \cite{ST87}.

	When no confusion may arise, we suppress the dependence of the function spaces on the underlying domain, writing for example $B^t_{pq}$ instead of $B^t_{pq}(\T^d)$. We also employ the same function space notation for vector fields $f=(f_1,\dots,f_d)$. For instance, $f \in H^s = (H^s)^{\otimes d}$ will mean each $f_i \in H^s$ and the norm on $H^s$ is $\|f\|_{H^s} = \sum_{i=1}^d \|f_i\|_{H^s}$. Similarly, $\|\nabla g\|_p = \sum_{i=1}^d \|\partial_{x_i} g\|_p$.

	We write $\| . \|$ for the Euclidean norm in $\R^d$, $\lesssim$, $\gtrsim$ and $\simeq$ to denote one- or two-sided inequalities up to multiplicative constants that may either be universal or `fixed' in the context where the symbols appear. We also write $a_+ = \max(a,0)$ and $a \vee b=\max(a,b)$ for real numbers $a,b$. The $\eps$-covering number of a set $\Theta$ for a semimetric $d$, denoted $N(\Theta,d,\eps)$, is the minimal number of $d$-balls of radius $\eps$ needed to cover $\Theta$.

\subsection{Diffusions with periodic drift and Bayesian inference}\label{sec:problem_setup}

Consider the SDE \eqref{model} with drift $b=\nabla B$, for a twice-continuously differentiable and one-periodic potential $B:\R^d\to\R$, that is $B(x+m) = B(x)$ for all $m\in \Z^d$,
\begin{equation}\label{model_B}
	dX_t = \nabla B (X_t) dt + dW_t, 
	\qquad X_0 = x_0 \in \R^d, 
	\quad t\geq 0.
\end{equation}
There exists a $d$-dimensional strong pathwise solution $X=(X_t=(X_t^1,\dots,X_t^d): \ t\ge0)$ with cylindrically defined law on the path space $C([0,\infty);\R^d)$; see, e.g., Chapters 24 and 39 in \cite{B11}. For $T>0$, let $X^T:=(X_t: \ 0\le t\le T)$ and denote by $P_B$ the law of $X^T$ on $C([0,T],\R^d)$. We omit the dependence on the initial condition $X_0 = x_0$ since it plays no role in our results.

	By periodicity, we can consider $B$ as a function on $\T^d$. In model \eqref{model_B}, the law $P_B$ depends on $B$ only through $b= \nabla B$ (see \eqref{Eq:LogLik} below), which is thus only identifiable up to an additive constant. We therefore choose to work with the unique equivalence class given by $\int_{\T^d}B(x) dx  = 0$, i.e.~$B \in \dot{L}^2(\T^d)$. Our goal is to estimate the drift $b = \nabla B:\R^d \to \R^d$ from an observed trajectory $X^T\sim P_B$. We will sometimes write $P_b$ when a technical result also applies to possibly non-gradient vector field drifts, but this will be clarified in each instance.

	The periodic model effectively restricts the diffusion to the bounded state space $\T^d$. More precisely, while the diffusion defined in \eqref{model_B} takes values on all of $\R^d$, its values $(X_t)$ modulo $\Z^d$ contain all the relevant statistical information about $\nabla B$ (note that $(X_t)$ will not be globally recurrent on $\R^d$). This allows us to define an \textit{invariant measure} on $\T^d$, since it holds that (arguing as in the proof of Lemma 6 in \cite{nicklray2020})
$$
	\frac{1}{T}\int_0^T \varphi(X_s) ds \to^{P_B} \int_{\T^d} \varphi d\mu_b 
	\qquad \text{as }T\to\infty, \quad \forall \varphi \in C(\T^d),
$$
where $\mu_b$ is a uniquely defined probability measure on $\T^d$ and we identify $\varphi$ with its periodic extension to $\R^d$ on the left-hand side. The measure $\mu_b$ thus inherits the usual probabilistic interpretation as the limiting ergodic average in \eqref{erginv}.

	Recall that the generator $L:H^2(\T^d) \to L^2(\T^d)$ of the possibly non-reversible diffusion from \eqref{model} is
\begin{equation}\label{Eq:generator}
	L_b = \frac{1}{2}\Delta + b.\nabla 
	= \frac{1}{2}\sum_{i=1}^d \frac{\partial^2}{\partial x_i^2} 
	+ \sum_{i=1}^d b_i(\cdot) \frac{\partial}{\partial x_i}.
\end{equation}
If $b=\nabla B$ for some potential $B\in C^2(\T^d)$, the corresponding invariant measure of the diffusion has probability density function 
\begin{equation}
\label{Eq:InvMeas}
	\mu_B(x) = \frac{e^{2B(x)}}{\int_{\T^d}e^{2B(y)}dy}, \qquad x\in \T^d,
\end{equation}
see p. 45-47 in \cite{bakry2014}. In particular, if we have a potential $B$, then we can recover $b$ from $\mu_B$ via $b = \nabla B = \tfrac{1}{2} \nabla \log \mu_B$, a connection we exploit in our proofs. In a slight abuse of notation, we write either $\mu_b$ or $\mu_B$ depending on the context when $b = \nabla B$, and we use $\mu_b$ for both the probability measure and its density function.

	The log-likelihood for $B\in C^2(\T^d)$ for our observation model is given by Girsanov's theorem (e.g., Section 17.7 in \cite{B11}):
\begin{equation}
\label{Eq:LogLik}
	\ell _T(B)
	:=
		\log\frac{ dP_B}{dP_0}(X^T)
	=
		-\frac{1}{2}\int_0^T\|\nabla B (X_t)\|^2dt + \int_0^T\nabla B(X_t).dX_t,
\end{equation}
where $P_0$ is the law of a $d$-dimensional Brownian motion $(W_t:0\le t\le T)$. We consider a Bayesian approach to the problem, assigning a (possibly $T$-dependent) prior $\Pi = \Pi_T$ to $B$, which for identifiability we assume is supported on $\dot{L}^2(\T^d) \cap C^2(\T^d)$. The posterior $\Pi(\cdot|X^T)$ then takes the form
\begin{equation}
\label{Eq:Posterior}
	d\Pi(B|X^T) = \frac{e^{\ell_T(B)}d\Pi(B)}{\int e^{\ell_T(B')}d\Pi(B')}.
\end{equation}
Note that this induces a prior and posterior for both $b = \nabla B$ and $\mu_B$. 
In the following, we study the concentration of the posterior about the ‘ground truth' gradient vector field $b_0=\nabla B_0$, assuming that the observation $X^T\sim P_{B_0}$ is generated according to the SDE \eqref{model_B} with $B=B_0$.

\subsection{Gaussian process priors}
\label{Sec:GP}

\subsubsection{Contraction rates}
\label{Sec:GP}

Gaussian priors are widely employed for diffusion models \cite{PPRS12,PSvZ13,RBO13,GS14,vWvZ16,vdMS17,vW2019,BRO18,nicklray2020} and we provide here theoretical guarantees for such priors. We consider a class of Gaussian process priors constructed from a base Gaussian probability measure $\Pi_W$, which we assume satisfies the following condition. We refer, e.g., to Chapter 2 in \cite{ginenickl2016} for definitions and terminology regarding the theory of Gaussian processes and measures.

\begin{condition}\label{GP_condition}
For $s>(d-1)\vee(1/2)$, let $\Pi_W=\Pi_{W,T}$ be a centred Gaussian Borel probability measure on the Banach space $C(\T^d)$ that is supported on a separable (measurable) linear subspace of $C^{(d/2+\kappa)\vee 2}(\T^d) \cap \dot{L}^2(\T^d)$ for some $\kappa>0$, and assume its reproducing kernel Hilbert space (RKHS) $(\H,\|\cdot\|_\H)$ is continuously embedded into the Sobolev space $H^{s+1}(\T^d)$.
\end{condition}

	The constant $\kappa>0$ above can be taken arbitrarily small. Examples of Gaussian processes priors satisfying Condition \ref{GP_condition} include the periodic Mat\'ern process and high-dimensional Gaussian series expansions, see Examples \ref{Ex:Matern}-\ref{Ex:GaussWav} below.

	To control the nonlinearity of the problem, we rescale the base Gaussian processes following ideas from the Bayesian inverse problem literature \cite{MNP20,AN19,GN20,NS20,MNP20b}.
Given a random draw $W\sim\Pi_W$ satisfying Condition \ref{GP_condition}, consider the following rescaled function:
\begin{equation}
\label{prior}
	B(x) := \frac{W(x)}{T^{d/(4s+2d)}}, \qquad x\in \T^d,
\end{equation}
whose law $\Pi=\Pi_T$ we take as the prior for the potential $B$. It follows that $\Pi$ is a centred Gaussian probability measure on $C(\T^d)$, with the same support and RKHS as $\Pi_W$, but with rescaled RKHS norm $\| h \|_{\H_B} = T^{d/(4s+2d)}\|h\|_\H$.

	The rescaling enforces additional regularisation in the induced posterior distribution for the invariant measure $\mu_B$, implying in particular a bound for $\|\mu_B\|_{C^{(d/2+\kappa)\vee 2}}$, needed to control the nonlinear map $B \mapsto \mu_B$ given in \eqref{Eq:InvMeas}. Such issues are commonly encountered in nonlinear inverse problems, where one often requires the posterior to place most of its mass on sets of bounded higher-order smoothness in order to use stability estimates. Note that such priors are special cases of the rescaled Gaussian process priors considered in several benchmark statistical settings in \cite{vVvZ07}.

\begin{theorem}\label{Prop:GaussRates}
Let $\Pi =\Pi_T$ be the rescaled Gaussian process prior for $B$ in \eqref{prior} with $W\sim\Pi_W$ satisfying Condition \ref{GP_condition} for some $s>(d-1)\vee(1/2)$, some $\kappa>0$, and RKHS $\H$. Suppose that $B_0\in H^{s+1}(\T^d)$ and that there exists a sequence $B_{0,T} \in \H$ such that $\|B_0 - B_{0,T}\|_{C^1} =O(T^{-s/(2s+d)})$ and $\|B_{0,T}\|_\H = O(1)$ as $T\to\infty$. Then for $M>0$ large enough, as $T\to\infty$,
$$
	P_{B_0} \Pi(B:\|\nabla B-\nabla B_0\|_2 \geq MT^{-s/(2s+d)} |X^T) \to 0.
$$
\end{theorem}

	The resulting posterior thus contracts about the truth at the usual nonparametric rate for $s$-regular functions in any dimension $d$. Since the law $P_B$ depends on the potential $B$ only through $\nabla B$ in \eqref{model_B}, it is natural to study recovery of the gradient vector field $\nabla B_0$. Recall that we make the identifiability assumption that the prior, and hence posterior, for $B$ is supported on $\dot{L}^2(\T^d)$ in Condition \ref{GP_condition}. For $B,B_0 \in \dot{L}^2(\T^d)$, the norm $\|\nabla B - \nabla B_0\|_2$ is in fact equivalent to the usual Sobolev norm $\| B - B_0\|_{H^1}$ by the Poincaré inequality (e.g. p. 290 in \cite{E2010}).

	Theorem \ref{Prop:GaussRates} requires that the true $B_0$ be approximable by elements $B_{0,T}$ of the RKHS of $W$ at a suitable rate, which reflects the notion of smoothness being modelled by the Gaussian process prior. One can think of this condition as `$B_0$ is $(s+1)$-smooth' in both the Sobolev and prior sense. For instance, if the Gaussian prior already models Sobolev smooth functions (e.g.~a Mat\'ern process prior - see Example \ref{Ex:Matern}), then this poses no additional conditions.

	The last theorem implies the same converge rate for the posterior mean estimator.

\begin{corollary}\label{Cor:PostMean1}
Let $\hat{B}_T =E^{\Pi_T}[B|X^T]$ be the posterior mean. Under the conditions of Theorem \ref{Prop:GaussRates}, as $T \to \infty$,
$$\|\nabla \hat{B}_T - \nabla B_0\|_2 = O_{P_{B_0}}(T^{-s/(2s+d)} ).$$ 
\end{corollary}

Under natural regularity conditions, for example finite-dimensional priors, the maximum a-posteriori (MAP) estimator is well-defined as the minimizer of an objective function \cite{DLSV13,PSvZ13}:
$$\hat{B}_T = \hat{B}_T(X^T) = \argmin_{B \in \H} \left( -\ell_T(B) + \tfrac{1}{2} T^{d/(2s+d)} \|B\|_{\H}^2 \right).$$
Since the posterior is also Gaussian by Lemma \ref{Lem:Conjugacy} below, the MAP estimator, when it exists, equals the posterior mean $E^{\Pi_T}[B|X^T]$. In this case, Corollary \ref{Cor:PostMean1} can be viewed as a convergence rate for a penalized maximum likelihood estimator with penalty equal to the squared prior RKHS-norm.

%
%

%

\begin{remark}[Minimax rates]\label{rem:minimax}
Although lower bounds have not been formally proven in the exact periodic diffusion setting studied here, results for closely related non-periodic diffusion models suggest that the minimax rates are of the usual order $T^{-s/(2s+d)}$ in all the cases we consider. For reversible diffusions on $\R^d$, lower bounds for pointwise and $L^2$-loss follow from the local asymptotic equivalence of the diffusion model with a sequence of Gaussian shift experiments  proved in \cite{DR07}, see also \cite{S15,S16}. Given this asymptotic equivalence, it is instructive to consider lower bounds in related statistical models. In particular, for $L^p$-loss, $1\le p\le2$, with $s$-smooth Sobolev truths as in Theorem \ref{Prop:PExpRates} below, the minimax rate of estimation is $n^{-s/(2s+d)}$ (where $n$ plays the role of $T$) in Gaussian white noise \cite{LMS97}, density estimation \cite{DJ96} and nonparametric regression \cite{DJ96}. Since we obtain posterior contraction rate $T^{-s/(2s+d)}$ for all our results, this suggests all the priors we consider here are minimax optimal for estimation. We do not pursue the extension of such lower bounds to the current periodic setting, since periodicity is mainly a technical assumption to simplify the underlying PDE arguments, see Section \ref{sec:nonperiod}.
\end{remark}

\subsubsection{Examples of Gaussian priors}
\label{Sec:GPEx}

We now provide concrete examples of Gaussian priors to which Theorem \ref{Prop:GaussRates} applies. As discussed in Section \ref{sec:problem_setup}, the potential $B$ is only identified up to an additive constant, which we without loss of generality select via $\int_{\T^d} B(x) dx = 0$. For series expansions, one can enforce this by setting the coefficient of $e_0 \equiv 1$ (Fourier basis) or $\Phi_{-10} \equiv 1$ (wavelet basis) equal to zero. For more general Gaussian processes, one can simply recenter the prior draws by $B \mapsto B - \int_{\T^d} B(x) dx$.

		A common choice for this problem is a mean-zero Gaussian process with covariance operator equal to an inverse power of the Laplacian \cite{PSvZ13,vWvZ16,vW2019}, for which posterior inference based on discrete data can be computed efficiently using a finite element method \cite{PPRS12} (note that in the continuous model considered here, Gaussian priors for $B$ are conjugate, see Lemma \ref{Lem:Conjugacy} below). Such priors can be defined via a Karhunen-Loève expansion in the Fourier basis and are equivalent to periodic Mat\'ern processes, see Section \ref{Sec:PerMatProc} in the Supplement \cite{supp}.

\begin{example}[Periodic Matérn process]\label{Ex:Matern}
	For $s+1>d/2+(d/2)\vee 2$, consider the base Gaussian prior
\begin{equation}
\label{Eq:MaternKLSeries}
	W(x)= (2\pi)^{d/2} \sum_{k\in\Z^d} 
	\frac{1}{(1+4\pi^2 \|k\|^2)^{(s+1)/2}} g_k e_k(x), 
	\qquad g_k\overset{\textnormal{iid}}{\sim} N(0,1),
	\quad x\in\T^d,
\end{equation}	
corresponding to the series expansion of a periodic Matérn process with smoothness parameter $s+1-d/2$ (cf.~Section \ref{Sec:PerMatProc} in the Supplement for details). By the Fourier series characterisation of Sobolev spaces, its RKHS $\H$ equals $H^{s+1}(\T^d)$ with equivalent RKHS norm $\|\cdot\|_\H \simeq \|\cdot\|_{H^{s+1}}$. Furthermore, $W$ defines a Borel random element in $C^{s+1-\frac{d}{2}-\eta}( \T^d)$ for all $\eta>0$, which is a separable linear subspace of $C^{(d/2+\kappa)\vee 2}( \T^d)$ for sufficiently small $\kappa,\eta>0$ if $s+1>d/2+(d/2)\vee 2$. Condition \ref{GP_condition} therefore holds for periodic Mat\'ern processes. We may thus apply Theorem \ref{Prop:GaussRates} to the periodic Mat\'ern base prior in \eqref{Eq:MaternKLSeries} and any $B_0\in H^{s+1}(\T^d)= \H$ with $s+1> d/2 + (d/2)\vee 2$ by taking the trivial sequence $B_{0,T}=B_0\in H^{s+1}(\T^d)$.
\end{example}

	Another common approach to prior modelling is to obtain a high-dimensional discretisation by a truncated Gaussian series expansion. We illustrate this considering a wavelet expansion for concreteness, but analogous results can be derived for any basis compatible with the Sobolev smoothness scales, such as the Fourier basis.

\begin{example}[Truncated Gaussian series]\label{Ex:GaussWav}
Let $\{\Phi_{l r}, l \ge-1,r=0,\dots,\max(2^{l d}-1,0)\}$ be a periodized Daubechies wavelet basis of $L^2(\T^d)$ as described in Section \ref{Sec:Basics}, and consider the base prior
$$
	W(x)=\sum_{l \le J}\sum_r 2^{- l (s+1)}g_{l r}\Phi_{l r}(x), \qquad g_{l r}
	\overset{\textnormal{iid}}{\sim} N(0,1), \quad x\in\T^d,
$$
for some $s>(d-1)\vee(1/2)$ and where $J = J_T\in\N$ satisfies $2^J\simeq T^{1/(2s+d)}$, which is usually the optimal dimension of a finite-dimensional model for $s$-smooth functions. However, this is the correct choice of truncation for estimating an $(s+1)$-smooth invariant measure, which is equivalent to estimating the potential $B$ by \eqref{Eq:InvMeas}, due to the slightly different bias-variance decomposition for diffusions, see for instance Corollary 1 in \cite{DR07}.

	The support of $\Pi_W$ equals the finite-dimensional approximation space $V_J$, which is a separable linear subspace of $C^{(d/2+\kappa)\vee2}(\T^d)$. Its RKHS $\H$ equals $V_J$ with norm
$$
	\| h\|_\H^2=\sum_{l \le J}\sum_r2^{2 l (s+1)} |\langle h,\Phi_{l r}\rangle_2|^2 
	= \|h\|_{H^{s+1}}^2,
	\qquad h\in V_J,
$$
so that  $\Pi_W$ satisfies Condition \ref{GP_condition}. For $B_0\in H^{s+1}(\T^d)\cap C^{s+1}(\T^d)$, the wavelet projections $B_{0,T} = P_JB_0\in V_J=\H$ satisfy $\|P_JB_0\|_\H\le \|B_0\|_{H^{s+1}}<\infty$ and $\|B_0-P_JB_0\|_{C^1}\lesssim 2^{-Js}\simeq T^{-s/(2s+d)}$. Theorem \ref{Prop:GaussRates} therefore applies with $\Pi_W$ a Gaussian wavelet series and all $B_0\in H^{s+1}(\T^d)\cap C^{s+1}(\T^d)$ with $s>(d-1)\vee(1/2)$.
\end{example}

	In Section \ref{sec:p_exp}, we extend the last result to truncated $p$-exponential series priors. For $p=2$, Theorem \ref{Prop:PExpRates} below shows that the above additional smoothness requirement $B_0\in C^{s+1}(\T^d)$ can be removed under the slightly stronger minimal smoothness assumption $s>d/2+(d/2)\vee 2$. For discussion on extending these results to adaptive priors, see Section \ref{sec:adaptation} below.

\subsubsection{Conjugacy of Gaussian priors}
\label{Sec:ConjGaussPriors}

For the continuous observation model $X^T = (X_t: 0\leq t \leq T)$, the likelihood \eqref{Eq:LogLik} is of quadratic form in the potential $B$, and hence Gaussian priors are conjugate as we now show. This parallels the known conjugacy property of Gaussian priors for the drift vector field $b$ \cite{PSvZ13}, which in our setting corresponds to Gaussianity of the posterior on $\nabla B$.

\begin{lemma}\label{Lem:Conjugacy}
Let $\Pi = \Pi_T$ be a centred Gaussian Borel probability measure on $L^2(\T^d)$ that is supported on $C^2(\T^d)\cap \dot L^2(\T^d)$. Then the posterior distribution \eqref{Eq:Posterior} is almost surely (under the law of the data $X^T$) Gaussian on  $L^2(\T^d)$.
\end{lemma}

The proof of Lemma \ref{Lem:Conjugacy} can be found in Section \ref{Sec:Conjugacy} of the Supplement. The $C^2(\T^d)$ condition is a standard assumption on $B$ to ensure the existence of a strong pathwise solution to the SDE \eqref{model_B}, see e.g.~\cite{B11}, and thus is natural in our setting. Conjugacy implies that the computation of the posterior mean in Corollary \ref{Cor:PostMean1}, as well as posterior sampling, is straightforward to implement in this model. Consider a discretisation step by a Karhunen-Loève (KL) truncation, taking as prior the random function
\begin{equation}
\label{Eq:TruncatedKL}
	B(x) = \sum_{k=1}^K \upsilon_k
	g_k h_k(x), \qquad \ g_k\overset{\textnormal{iid}}{\sim} N(0,1), \quad x\in\T^d,
\end{equation}
for some fixed $K\in\N$, scaling coefficients $\upsilon_k>0$, and some ‘basis' functions $(h_k, \ k\in\N)\subset C^2(\T^d)\cap \dot L^2(\T^d)$ (e.g., the Fourier or wavelet bases). Identifying a function $B = \sum_{k=1}^K B_k h_K$ with its coefficient vector $\mathbf B = (B_1,\dots, B_K)^T\in \R^K,$ a standard conjugate computation yields
\begin{equation}
\label{Eq:TruncatedPosterior}
	\mathbf B | X^T \sim N\left((\Sigma+\Upsilon^{-1})^{-1}
		\mathbf H, (\Sigma+\Upsilon^{-1})^{-1}\right),
\end{equation}
where $\Upsilon = \textnormal{diag}(\upsilon_1^2,\dots,\upsilon_K^2)$ is a $K \times K$ diagonal matrix and
$$
	\Sigma = \left[ \int_0^T \nabla h_k
	(X_t).\nabla h_{k'}(X_t) dt \right]_{k,k'}\in\R^{K\times K}, \
	\mathbf{H} = \left[\left(\int_0^T \nabla h_k(X_t). dX_t\right)_{k=1}^K\right]^T
	\in \R^K.
$$
Additional details can be found in Section \ref{Sec:Conjugacy} of the Supplement. For concrete basis choices for the KL-expansion \eqref{Eq:TruncatedKL}, $\Sigma$ and $\mathbf H$ can be computed from the data, allowing direct posterior sampling according to \eqref{Eq:TruncatedPosterior}, see for instance Algorithm 2.1 in \cite{RW06} for implementation details. Gaussian conjugacy no longer holds for the more realistic discrete data setting, where more advanced sampling techniques must be employed, see for instance \cite{PPRS12,RBO13,vdMS17b,BRO18}.

\subsection{$p$-exponential priors}
\label{sec:p_exp}

We next consider modelling the potential function $B$ using the class of heavier-tailed $p$-exponential priors \cite{ADH20},  known in the inverse problems literature as Besov priors \cite{LSS09}. These priors are constructed via random basis expansions, assigning i.i.d.~random coefficients distributed according to the probability density function
$$
	f_p(x)\propto e^{-\frac{|x|^p}{p}}, \qquad x\in\R, \quad p\in [1,2].
$$
This generalizes the series construction of Gaussian priors ($p=2$), allowing heavier-tailed random coefficients for $p< 2$, while preserving a log-concave structure favourable to computation, see Remark \ref{rem:pExpComp}. The class includes products of Laplace distributions ($p=1$), which have recently received significant interest in the Bayesian inverse problem community \cite{KLNS12,DHS12,R13,BH15,DS17,ABDH18,AW21} due to their edge-preserving and sparsity-promoting properties. For Laplace priors, these advantages stem from the $\ell^1$-type penalty induced by the prior, which promotes sparse reconstructions that have been observed to perform better in practice for the recovery of spatially-irregular, blocky structures such as images, see e.g.~\cite{REtAl06,LP15,KLNS12,BG15,JP16} and references therein. It is therefore of interest to provide theoretical guarantees for these methods which are employed in practice.

	We consider $p$-exponential priors defined via a truncated wavelet expansion. Let $\{\Phi_{l r}, l \ge-1,r=0,\dots,\max(2^{l d}-1,0)\}$ be a periodized Daubechies wavelet basis of $L^2(\T^d)$ as described in Section \ref{Sec:Basics}. For $p\in[1,2]$ and $s \geq 0$, consider the \textit{$p$-exponential measure} \cite{ADH20} $\Pi_W = \Pi_{W,T}$ arising as the law of the random function
\begin{equation}
\label{Eq:pExpBasePrior}
	W(x)=\sum_{l =0 }^J\sum_r
	2^{- l \left(s+1+\frac{d}{2}-\frac{d}{p}\right)} \rho_{l r} \Phi_{l r}(x),
	\qquad x\in\T^d,
\end{equation}
where $\rho_{lr} \overset{\textnormal{iid}}{\sim} f_p$ are $p$-exponential random variables with $f_p$ defined above, and where the truncation level $J = J_T\in\N$ satisfies $2^J\simeq T^{1/(2s+d)}$ as $T\to\infty$. Note that this class includes both the product Laplace prior ($p=1$) and Gaussian series prior ($p=2$). For identifiability, the wavelet coefficient corresponding to $\Phi_{-10}\equiv 1$ is again set to zero to enforce the zero-integral condition, so that $W\in \dot{L}^2(\T^d)$ almost surely. Similar to the Gaussian priors considered in the previous section, we introduce a suitable scaling of $W\sim\Pi_W $, taking as prior $\Pi=\Pi_T$ for $B$ the law of
\begin{equation}
\label{Eq:pExpPrior}
	B(x)=\frac{W(x)}{\big(T^\frac{d}{2s+d}\big)^\frac{1}{p}},
	\qquad x\in\T^d.
\end{equation}
This is the correct scaling, since scaling at a different rate yields suboptimal contraction rates even in the simpler Gaussian white noise model \cite[Proposition 5.8]{ADH20}. The next theorem shows that the resulting posterior contracts about the truth at the nonparametric rate $T^{-s/(2s+d)}$ in any dimension $d$.


%

\begin{theorem}\label{Prop:PExpRates}
	Let $\Pi=\Pi_T$ be the scaled $p$-exponential truncated series prior \eqref{Eq:pExpPrior}, where $W$is as in \eqref{Eq:pExpBasePrior} for $s>d/p+(d/2)\vee 2$ and $p\in[1,2]$. Suppose that $B_0\in H^{s+1}(\T^d)$. Then for $M>0$ large enough, as $T\to\infty$,
$$
	P_{B_0}\Pi\big(B : \|\nabla B - \nabla B_0\|_p
	\ge M T^{-s/(2s+d)}\big|X^T\big)
	\to 0.
$$
\end{theorem}

	As in Theorem \ref{Prop:GaussRates}, the same rate of contraction is obtained for $\|B - B_0\|_{W^{1,p}}$ under the identifiability assumption $B,B_0 \in \dot{L}^2(\T^d)$ using the Poincar\'e inequality. For $p=2$, Theorem \ref{Prop:PExpRates} implies the result in Example \ref{Ex:GaussWav} for truncated Gaussian wavelet series priors, but removes the additional assumption that $B_0\in C^{s+1}(\T^d)$ under a slightly stronger minimal smoothness assumption on $s$. The rate $T^{-s/(2s+d)}$ matches the minimax rate we expect in this setting, see Remark \ref{rem:minimax}.

The last theorem can be used to obtain the same convergence rate for the corresponding posterior mean using ideas from \cite{MNP20}. Unlike in the Gaussian case, the posterior mean does not in general equal the MAP estimate, and so does not necessarily inherit the interpretation as a penalized maximum likelihood estimator.

\begin{corollary}\label{Cor:PostMean2}
Let $\hat{B}_T =E^{\Pi_T}[B|X^T]$ be the posterior mean. Under the conditions of Theorem \ref{Prop:PExpRates}, as $T \to \infty$,
$$\|\nabla \hat{B}_T - \nabla B_0\|_p = O_{P_{B_0}}(T^{-s/(2s+d)} ).$$ 
\end{corollary}

\begin{remark}[Loss functions]
Our proof approach requires that the induced prior on the invariant measure $\mu_B$ has enough regularity with respect to the chosen loss function not to induce too large a bias. Specifically, for $L^q$-loss, our general contraction Theorem \ref{thm:contraction_b} below requires that $\|\mu_B - P_J \mu_B\|_{W^{1,q}} \lesssim T^{-s/(2s+d)}$ with high probability under the prior, where $W^{1,q}$ is the usual Sobolev space and we recall $P_J$ is the $L^2$-projection operator onto the wavelet space $V_J$ in \eqref{V_J}. Draws from $p$-exponential priors have sample smoothness, as measured by the support and concentration properties of the prior (cf.~Lemma \ref{lem:prior_prob_event}), reflected in terms of $L^p$-type regularity, so that we verify the bias condition when this dominates the loss function, i.e. when $q \leq p$. For Gaussian priors ($p=2$), we can therefore employ $L^2$-loss in Theorem \ref{Prop:GaussRates}. Note that for any $p\in[1,2]$, Theorem \ref{Prop:PExpRates} implies that the posterior contracts at the nonparametric rate $T^{-s/(2s+d)}$ in $L^1$-loss.
\end{remark}

\begin{remark}[Posterior sampling]\label{rem:pExpComp}
For the $p$-exponential prior \eqref{Eq:pExpPrior}, the posterior density for any $B=\sum_{l\le J, r} B_{lr}\Phi_{lr}\in V_J$ takes the form $d\Pi(B|X^T)$ $\propto e^{-\Psi_T(B)}$, where
\begin{align*}
	\Psi_T( B) 
	&= -\ell_T(B)
	+ \tfrac{1}{p} T^\frac{d}{2s+d}  \| B \|^p_{B^{s+1}_{pp}} \\
	&=
		\frac{1}{2}\sum_{l\le J,r}\sum_{l'\le J,r'}B_{lr}B_{l'r'}\left[\int_0^T
		\nabla\Phi_{lr} (X_t).
		\nabla\Phi_{l'r'} (X_t)dt \right]\\
	&\quad
		- \sum_{l\le J,r}B_{lr} \left[ \int_0^T 
		\nabla\Phi_{lr} (X_t). dX_t\right]
		+ \frac{1}{p} T^\frac{d}{2s+d} \sum_{l\le J,r} 
		2^{pl\left(s+1+\frac{d}{2}-\frac{d}{p}\right)}| B_{lr} |^p .
\end{align*}
Since $p\ge1$, $\Psi_T( B)$ is, given the data $X^T$, a convex functional of $B$, implying that the posterior is log-concave. Approximate posterior sampling is thus feasible using MCMC algorithms for log-concave distributions. In particular, non-asymptotic convergence guarantees suitable for high-dimensional settings have been derived for gradient-based Langevin Monte Carlo methods \cite{D17,DM17,DM19}, as well as for the relevant non-smooth case $p=1$ using suitable proximal regularisation of the log-posterior density \cite{P15,DMP18}.
\end{remark}

\subsection{A general contraction theorem for multi-dimensional diffusions with gradient vector field drift}
\label{Sec:GenTheo}

The results for Gaussian and $p$-exponential priors presented in the preceding sections are based on the following general contraction rate theorem for the drift $b = \nabla B$. We employ the testing approach of \cite{ghosal2000}, which requires the construction of suitable tests with exponentially decaying type-II errors. This has been done in \cite{vdMeulen2006} for the `natural distance' for this model, which is an observation-dependent `random Hellinger semimetric'. In dimension $d=1$, this can be related to the $L^2$ distance using the theory of diffusion local times, something which is unavailable in dimension $d> 1$. We instead directly construct plug-in tests based on the concentration properties of preliminary estimators following ideas from the i.i.d.~density estimation model \cite{GN2011}. In the present multi-dimensional diffusion setting, suitable estimators can be obtained by exploiting the one-to-one connection between the potential $B$ and the invariant measure $\mu_B$ given by \eqref{Eq:InvMeas}.

\begin{theorem}\label{thm:contraction_b}
Let $q\in[1,2]$, $J=J_T\in \mathbb{N}$, $\eps_T\to 0$ and $\xi_T \to 0$ satisfy $2^J\to \infty$ , $T\eps_T^2 \to \infty$ and $T^{-1/2}2^{Jd/2} + \eps_T =O( \xi_T)$ as $T\to\infty$, and let $\Pi=\Pi_T$ be priors for $B$ supported on the Banach space $C^2(\T^d)$. Assume further that
\begin{equation}
\label{Eq:QuantCond}
	2^{J[d/2+\kappa+(d/2+\kappa-1)_+]}\eps_T=O(1) 
	\qquad \text{and} \qquad  T^{-1/2} 2^{J[d+\kappa+(d/2+\kappa-1)_+]}=O(1)
\end{equation}
for some $\kappa>0$. Consider sets
\begin{align}
\label{Eq:LambdaTProp}
	\Lambda_T  \subseteq \left\{ \mu  : \int_{\T^d} \mu(x)dx = 1, ~ \mu(x) 
	\geq \delta , ~ \|\mu\|_{C^{(d/2+\kappa)\vee 2}} \leq m
	,~  \|\mu-P_J\mu\|_{W^{1,q}} \leq C_\Lambda \xi_T \right\}
\end{align}
for some $\delta,C_\Lambda,m>0$ and define $\Theta_T = \{ B : \nabla B = \tfrac{1}{2} \nabla \log \mu \text{ for some } \mu\in \Lambda_T\}$. Let $B_0$ be the true potential and suppose that $\mu_0 = \frac{e^{2B_0}}{\int_{\T^d}e^{2B_0(y)}dy}$ satisfies $\|\mu_0-P_J\mu_0\|_{W^{1,q}} \leq D_0\xi_T$.
Suppose further that
\begin{itemize}
\item[(i)] $\Pi(\Theta_T^c) \leq e^{-(C+4)T\eps_T^2}$,
\item[(ii)] there exist deterministic sets $\mathcal{SB}_T$ for $B$ with $\Pi( \mathcal{SB}_T) \geq e^{-CT\eps_T^2}$ and
\begin{equation}\label{SB}
	P_{B_0} \left(\sup_{B\in \mathcal{SB}_T}  
	\int_0^T  \|\nabla B(X_s)-\nabla B_0(X_s)\|^2 ds
	 \leq T \eps_T^2 \right) \to 1.
\end{equation}
\end{itemize} 
Then for $M>0$ large enough, as $T\to \infty$,
$$
	P_{B_0} \Pi(B:\|\nabla B-\nabla B_0\|_q \geq M\xi_T|X^T) \to 0.
$$
\end{theorem}

\begin{remark}[Small ball probability]\label{rem:SB}
One can always take $\mathcal{SB}_T = \{ B: \|\nabla B -\nabla B_0\|_\infty^2 $ $ \leq \eps_T^2\}$, in which case \eqref{SB} holds automatically. However, for truncated wavelet series priors, this leads to unnecessary smoothness conditions on the true underlying function $B_0$, which can be avoided by instead taking $\mathcal{SB}_T = \{ B: \|\nabla B-\nabla B_{0}\|_{L^2(\mu_0)}^2 \leq \eps_T^2\}$ and verifying \eqref{SB} directly (cf.~Lemma \ref{Lemma:AnotherSmallBall}). The set $\mathcal{SB}_T$ will typically be a norm ball, and by `deterministic' we mean that the norm is non-random. This is required in order to apply certain martingale arguments in the proof of Theorem \ref{thm:contraction_b}.
\end{remark}

	Beyond the `usual' conditions arising from the testing approach \cite{ghosal2000}, the main additional assumption in the last theorem is that the prior puts most of its mass on a set of potentials $\Theta_T$, where the corresponding invariant measures $\mu\in\Lambda_T$ can be well approximated by their wavelet projections $P_J \mu$ (cf.~the last inequality in \eqref{Eq:LambdaTProp}). If this is the case, it suffices to study the deviations of the wavelet projection estimator $\hat{\mu}_T$ about its expectation $P_J \mu$. We then use results from empirical process theory, martingale theory and PDEs in order to obtain suitable concentration inequalities, and hence exponential probability bounds for the type-II errors of suitable tests.

	A significant additional difficulty in carrying out this program is the nonlinearity of the map $B \mapsto \mu_B$ given by \eqref{Eq:InvMeas}. One can relate the error of $\nabla B - \nabla  B_0$ to that of $\mu_B- \mu_{B_0}$ by a type of stability estimate of the form $\|\nabla B-\nabla B_0\|_{q} \lesssim \|\mu_B - \mu_{B_0}\|_{W^{1,q}}$, see the proof of Lemma \ref{lem:W11_exp} below. However, without controlling the norm $\|\mu\|_{C^{(d/2+\kappa)\vee 2}}$, the constant for this stability estimate grows rapidly, rendering it unusable in the proofs. This reinforces the connection of recovery of the gradient vector field $\nabla B_0$ in the diffusion model \eqref{model_B} with nonlinear inverse problems, where similar phenomena are often encountered, and indeed motivated the use of rescaled priors to overcome these considerable technical challenges \cite{MNP20,AN19,GN20,NS20,MNP20b}.

\section{Generalizations and extensions}\label{sec:extensions}

\subsection{Adaptation}\label{sec:adaptation}

The results in this paper are non-adaptive since the rescaled Gaussian and $p$-exponential priors require correct calibration based on the (typically unknown) regularity $s$ of the truth to achieve optimal contraction rates. A natural Bayesian approach to adaptation is to assign a hyperprior to $s$, as is studied for example in the one-dimensional diffusion setting using local-time techniques in \cite{vWvZ16}, which we recall are not available in dimension $d\geq 2$.

Our alternative approach of using plug-in tests based on estimators satisfying good concentration inequalities requires the posterior to concentrate on sets of bounded $\|\mu\|_{C^{(d/2+\kappa)\vee 2}}$-norm in order to control the constant in the stability-type estimate $\|\nabla B-\nabla B_0\|_{q} \lesssim \|\mu_B - \mu_{B_0}\|_{W^{1,q}}$. We prove this for non-adaptive priors using the rescaling combined with precise isoperimetric inequalities for Gaussian and $p$-exponential measures, but are unable to verify this for hierarchical constructions. This additional difficulty stems from the nonlinearity of the problem rather than the choice of prior, and indeed establishing adaptation is a general open problem for Bayesian nonlinear inverse problems. We expect that adaptation is possible in our model, but proving this will require novel ideas to deal with nonlinear maps.

\subsection{Non-periodic potentials}\label{sec:nonperiod}

We restricted here to periodic potentials $B$ to ensure recurrence and mixing of the diffusion and simplify certain elliptic PDE arguments. The fundamental issue is that to be in the ergodic setting, where our statistical analysis is relevant, one must have enough recurrence to essentially confine the diffusion to a bounded set. Options include enforcing a suitable drift condition to prevent the particle escaping to infinity or restricting to bounded domains. We discuss some of the technical challenges involved in extending our results to non-periodic potentials in these two cases.

A key step in our approach is establishing concentration inequalities for empirical processes (Proposition \ref{prop:sup_prob}), which follow from martingale techniques and properties of the generator $L_{\nabla B}$ defined in \eqref{Eq:generator}, written $L_B$ for simplicity, which is an elliptic second order partial differential operator. In particular, we rely on a standard regularity estimate for solutions to the Poisson equation of the form
\begin{equation}\label{eq:Poisson}
\|u_f\|_{H^t} \lesssim \|f\|_{H^{t-2}}, \qquad \text{where} \qquad L_{B} u_f = f, \qquad \qquad f\in \dot{L}_{\mu_B}^2(\T^d),
\end{equation}
with uniform constants over certain sets of potentials  $B$, see Lemma \ref{regest} in the supplement \cite{supp} and Section 6 in \cite{nicklray2020}.
Periodicity ensures we can restrict to the torus $\T^d$ when studying the Poisson equation, in which case one can
use Fredholm theory and properties of Fourier series to establish existence and regularity of solutions to \eqref{eq:Poisson}, see Chapter II.3 in \cite{BJS64}. 
In particular, the solution map $L_B^{-1}:H^{t-2}(\T^d) \to H^t(\T^d)$ is a compact operator, a fact which no longer holds true when one considers unbounded domains. This lack of compactness is a key difficulty, which is why standard elliptic PDE techniques do not extend straightforwardly to unbounded sets. A possible alternative approach is to extend the probabilistic ideas of Pardoux and Veretennikov \cite{PV01}, who study existence and regularity of solutions to \eqref{eq:Poisson} using a stochastic solution representation. However, this will require substantial and new refined elliptic PDE results that are beyond the scope of this work.

To use existing PDE arguments involving compactness, one must therefore restrict to bounded domains and pick suitable boundary conditions. Dirichlet boundary conditions correspond to the particle being killed upon exiting the domain, and are thus inappropriate for our ergodic ($T\to\infty$) setting. Neumann boundary conditions correspond to reflecting the particle upon hitting the boundary, and one can in principle use more involved elliptic PDE techniques to obtain regularity estimates. However, additional probabilistic tools are then needed to deal with the boundary reflection, such as local times to use stochastic calculus as we do here. A third option is periodicity, which simplifies the boundary technicalities without changing the diffusion's behaviour in the interior of the domain. In all cases, one requires the diffusion to have a uniform notion of strong recurrence with the various models largely differing in how they model the boundary from a technical perspective. Periodicity represents a  suitable compromise preserving the main statistical ideas and implications without being overly encumbered by technicalities.

\subsection{Models with non-constant diffusivity}\label{sec:diffusivity}

We comment on some implications for likelihood-based estimation in generalizations of the models \eqref{model} and \eqref{model_B}. Consider observing a continuous trajectory $X^T = (X_t: 0 \leq t \leq T)$ from diffusion dynamics of the form
$$	dX_t = b(X_t) dt + \Sigma^{1/2} (X_t) dW_t,  \qquad X_0 = x_0 \in \R^d,  \quad t\geq 0,$$ 
with non-constant local variance (diffusion or volatility matrix) $\Sigma:\R^d \to \R^d \times \R^d$. For general positive-definite $\Sigma(\cdot)$, this model is not identifiable and so we restrict to the case of scalar local variance $\Sigma(X_t) = \sigma(X_t) I_d$, where $\sigma:\R^d \to [0,\infty)$ and $I_d$ is the $d\times d$ identity matrix.

	In this model, one can exactly recover the quadratic variation process of any component of $X$, namely $([X^i]_t = \int_0^t \sigma(X_s)^2 ds: 0 \leq t \leq T)$, $i=1,\dots,d$, and hence $\{ \sigma(X_t): 0 \leq t \leq T\}$ is perfectly identified from the data $X^T$. In the well-studied scalar case $d=1$, this corresponds to knowledge of $\{ \sigma(x): \inf_{t\in[0,T]} X_t \leq x \leq \sup_{t\in[0,T]} X_t\}$ and yields the conventional wisdom in the statistical diffusion literature with continuous data that one can treat $\sigma(\cdot)$ as known, usually taking $\sigma(x) \equiv 1$ for simplicity as we do here. For dimension $d\geq 2$, while we still perfectly identify $\sigma$ \textit{along the trajectory} of the diffusion, this trajectory now has zero Lebesgue measure in $\R^d$ (note that for $d\geq 2$, the diffusion process will not be recurrent). Thus the diffusivity function $\sigma(\cdot)$ is a non-trivial parameter and there may still be statistical interest in modelling it.

	Consider placing a prior on $\sigma(\cdot)$ and let $P_{b,\sigma} = P_{b,\sigma}^T$ be the law of the above process. Girsanov's theorem (e.g., Section 17.7 in \cite{B11}) implies that the two measures $P_{b,\sigma} = P_{b',\sigma'}$ are singular unless $\sigma = \sigma'$.  The likelihood $e^{\ell_T(b,\sigma)}$ is thus zero unless $\sigma$ \textit{exactly} matches the observed diffusivity along the trajectory. Therefore, for a Bayesian posterior to be well-defined, the prior must assign positive probability to $\{ \sigma(\cdot): \sigma(X_t) = ( \tfrac{d}{dt}[X^i]_t)^{1/2} , ~ 0\leq t \leq T\}$. i.e.~it must be conditioned to match these values along the trajectory. Since the random trajectory $X^T$ typically has fractal-like behaviour, this is a highly non-standard and non-trivial prior construction. It is a somewhat unusual feature of this model that any prior for $\sigma(\cdot)$ must heavily rely on the observed data $X^T$. Similar considerations apply to other likelihood-based procedures, such as maximum likelihood estimation.

	The above features stem from the continuous observation model and do not occur in the more realistic low frequency discrete observation model, where estimation of $(b,\sigma)$ is an ill-posed inverse problem, see \cite{GHR2004,NS17} for the scalar case $d=1$. While both cases are interesting mathematically, the continuous and discrete models are fundamentally different problems with regards to estimating the diffusivity $\sigma(\cdot)$. We are unaware of any results concerning minimax rates in dimension $d\geq 2$ in the low frequency setting.

\section{Proof of Theorem \ref{thm:contraction_b}}
\label{Sec:ProofMainTheo}

We employ the general testing approach for non-i.i.d.~sampling models \cite{ghosal2007} combined with tools from the diffusion setting \cite{vdMeulen2006}. In order to construct suitable plug-in tests, we extend ideas from the i.i.d.~density model \cite{GN2011} to the multi-dimensional diffusion setting with drift arising as a gradient vector field $b = \nabla B$.

	We start with the following contraction rate theorem, which applies also to non-reversible diffusions, based on the existence of abstract tests. In a slight abuse of notation, denote by $P_b$ the law of $(X_t: 0 \leq t \leq T)$ from model \eqref{model}, i.e.~we do not assume $b = \nabla B$ in the next result.

\begin{theorem}\label{thm_contraction_tests}
	Let $d_T$ be a semimetric on the parameter space $\mathcal{H} \subseteq C^1(\T^d)$ for the drift $b$ and let $\Pi = \Pi_T$ be priors for $b$. Let $\eps_T\to 0$ satisfy $\sqrt{T}\eps_T \to \infty$, let $\xi_T \to 0$, $\mathcal{H}_T \subseteq \mathcal{H}$ and let $\phi_T$ be a sequence of test functions satisfying
\begin{align*}
	P_{b_0} \phi_T \to 0, \qquad \qquad 
	\sup_{b \in \mathcal{H}_T :  d_T(b,b_0) \geq D \xi_T} P_b (1-\phi_T) \leq Le^{-(C+4)T\eps_T^2}
\end{align*}
for some $C,D,L>0$ with $\Pi(\mathcal{H}_T^c) \leq e^{-(C+4)T\eps_T^2}$. Suppose further that there exist deterministic sets $\mathcal{SB}_T \subseteq \mathcal{H}$ with $\Pi(\mathcal{SB}_T) \geq e^{-CT\eps_T^2}$ and
$$P_{b_0} \left(\sup_{b\in \mathcal{SB}_T}  \int_0^T \|b(X_s)-b_0(X_s)\|^2 ds \leq T \eps_T^2 \right) \to 1,$$
where $b_0$ is the true drift function. Then for $M>0$ large enough, as $T\to\infty$,
$$
	P_{b_0} \Pi (b:d_T(b,b_0) \geq M\xi_T|X^T) \to 0.
$$
\end{theorem}

	The proof of Theorem \ref{thm_contraction_tests} follows similarly to results in \cite{ghosal2000,vdMeulen2006} and is deferred to Section \ref{Sec:ProofTestTheo} of the Supplement. The required tests are contained in the next lemma, whose proof can be found in Section \ref{sec:tests} below.

\begin{lemma}\label{lem:tests}
	Let $q\in[1,2]$, $J=J_T\in \mathbb{N}$, $\eps_T\to 0$ and $\xi_T \to 0$ satisfy $2^J\to \infty$, $T\eps_T^2 \to \infty$ and $T^{-1/2}2^{Jd/2} + \eps_T =O( \xi_T)$ as $T\to\infty$. Assume further that 
$$
	2^{J[d/2+\kappa+(d/2+\kappa-1)_+]}\eps_T=O(1) \qquad \text{and} \qquad 
	T^{-1/2} 2^{J[d+\kappa+(d/2+\kappa-1)_+]}=O(1)
$$
for some $\kappa>0$. Consider sets
\begin{align*}
	\Lambda_T  \subseteq \left\{ \mu  : \int_{\T^d} \mu(x)dx = 1,
	 ~ \mu(x) \geq \delta , ~ \|\mu\|_{C^{(d/2+\kappa)\vee 2}} \leq m,
	 ~  \|\mu-P_J\mu\|_{W^{1,q}} \leq C_\Lambda \xi_T \right\}
\end{align*}
for some $\delta,C_\Lambda,m>0$ and define $\Theta_T = \{ B : \nabla B = \tfrac{1}{2} \nabla \log \mu \text{ for some } \mu\in \Lambda_T\}$. Let $B_0$ and $\mu_0$ be the true potential and invariant measure, respectively, and assume that $\|\mu_0-P_J\mu_0\|_{W^{1,q}} \leq D_0\xi_T$ for some $D_0>0$. Then for any $M>0$, there exist tests $\phi_T$ such that for $D=D(q,\delta,m,C_\Lambda,D_0,M)>0$ large enough,
\begin{align*}
	P_{B_0} \phi_T \to 0, \qquad \qquad \sup_{B \in \Theta_T:
	\|\nabla B-\nabla B_0\|_q \geq D \xi_T} P_B (1-\phi_T) 
	\leq 4e^{-MT\eps_T^2}.
\end{align*}
\end{lemma}

\begin{proof}[Proof of Theorem \ref{thm:contraction_b}]
The conclusion of Theorem \ref{thm:contraction_b} follows by applying Theorem \ref{thm_contraction_tests} with the set $\mathcal{H}_T = \{b = \nabla B: B \in \Theta_T \}$, the distance $d_T(f,g) = \|f-g\|_q$ and the tests $\phi_T$ constructed in Lemma \ref{lem:tests}.
\end{proof}

\subsection{A concentration of measure result for empirical processes}
\label{Sec:EmpProc}

The following concentration inequality is a key technical tool in the proof of Theorem \ref{thm:contraction_b}, providing uniform stochastic control of functionals of the (possibly non-reversible) diffusion process \eqref{model}. It is based on a chaining argument for stochastic processes with mixed tails (cf.~Theorem 2.2.28 in Talagrand \cite{T14} and Theorem 3.5 in Dirksen \cite{D15}). We again write $P_b$ for the law of $X^T$ to emphasise that we do not assume $b = \nabla B$ in the next result. Recall the notation $\dot L^2 = \{f \in L^2 : \int_{\T^d} fdx = 0\} $ and $\dot{L}^2_\mu = \{ f \in L^2 : \int_{\T^d} fd\mu = 0\}$.

\begin{proposition}\label{prop:sup_prob}
	Suppose $b\in C^{(d/2+\kappa)\vee 1}(\T^d)$ for some $\kappa >0$, and let $\mF_T\subset V_J \cap \dot{L}_{\mu_b}(\T^d)$ for $J$ satisfying $2^{J[d/2+\kappa + (d/2+\kappa-1)_+]} \lesssim \sqrt{T}$. Define the empirical process
$$
	\G_T(f) := \frac{1}{\sqrt{T}}\int_0^Tf(X_s)ds , \qquad f\in \mathcal{F}_T,
$$
and let $D_{\mF_T}:= \textnormal{dim}(\mF_T)$ and $|\mF_T|_{H^{-1}}:=\sup_{f\in \mF_T} \|f\|_{H^{-1}}$. Then for all $T \geq \eta > 0$ and $x\geq 1$,
$$
	P_b\left( \sup_{f\in \mF_T} |\G_T(f)| \geq C|\mF_T|_{H^{-1}}
	\left\{ D_{\mF_T}^{1/2}+\sqrt{x} + T^{-1/2} 2^{J[d/2+\kappa+(d/2+\kappa-1)_+]}
	(D_{\mF_T}+ x) \right\} \right)\leq 2e^{-x},
$$
where $C$ depends on $d,\kappa,\eta$ and upper bounds for $\|b\|_{B_{\infty\infty}^{|d/2+\kappa-1|\vee 1}}$ and $\|\mu_b\|_\infty$.

\proof
	We first note that since $b \in C^1$, a corresponding unique invariant probability measure $\mu = \mu_b$ indeed exists by Proposition 1 of \cite{nicklray2020}.
For $f\in \dot{L}^2_\mu \cap V_J \subset L^2_\mu \cap H^{d/2+\kappa}$ and $b\in C^{d/2+\kappa}$, by Lemma \ref{regest} in the Supplement and the Sobolev embedding theorem, the Poisson equation $L_b u = f$ has a unique solution $L_b^{-1}[f]\in \dot L^2 \cap H^{d/2+\kappa+2}\subset C^2$ satisfying $L_bL_b^{-1}[f] = f$. For such $f$, we may thus define 
\begin{equation*}
\begin{split}
	Z_T(f) 
	&:=  
		\int_0^T \nabla L_b^{-1}[f](X_s).dW_s \\
	& =
		 L_b^{-1}[f](X_T) - L_b^{-1}[f](X_0) -  \int_0^T L_bL_b^{-1}[f](X_s) ds \\
	& = 
		L_b^{-1}[f](X_T) - L_b^{-1}[f](X_0) -  \sqrt{T}\G_T[f],
\end{split}
\end{equation*}
where we have used It\^o's lemma (Theorem 39.3 of \cite{B11}). Thus for $\mF_T \subset \dot{L}^2_\mu \cap H^{d/2+\kappa}$,
\begin{equation}\label{subgauss_process}
\sup_{f\in \mathcal{F}_T} |\G_T[f]| \leq \frac{1}{\sqrt{T}} \sup_{f\in \mathcal{F}_T}|Z_T(f)| + \frac{2}{\sqrt{T}} \sup_{f\in \mathcal{F}_T} \|L_b^{-1}[f]\|_\infty.
\end{equation}
We derive a concentration inequality for $\sup_f |Z_T(f)|$ and hence for $\sup_f |\G_T[f]|$.

Recall Bernstein's inequality for continuous local martingales (p.~153 of \cite{revuz1999}): if $M$ is a continuous local martingale vanishing at 0 with quadratic variation $[M]$, then for any stopping time $T$ and any $y,K>0$,
\begin{align}
\label{Bernstein}
	\Pr \left( \sup_{0\leq t \leq T}|M_t| \geq y, [M]_T \leq K \right) 
	\leq 2e^{-\frac{y^2}{2K}}.
\end{align}
For fixed $f$, $(Z_T(f):T\geq 0)$ is a continuous square integrable local martingale with quadratic variation $[Z_\cdot(f)]_T = \int_0^T \|\nabla L_b^{-1}[f](X_s)\|^2 ds$. Applying Bernstein's inequality,
\begin{equation}\label{QV_split}
\begin{split}
	P_b\left( |Z_T(f)| \geq x \right)  
	& \leq 
		P_b\left(|Z_T(f)| \geq x, [Z_\cdot(f)]_T \leq K_T(f)\right) 
		+ P_b\left([Z_\cdot(f)]_T > K_T(f)\right)\\
	& \leq 
		2\exp \left(-\tfrac{x^2}{2K_T(f) } \right)+ P_b([Z_\cdot(f)]_T > K_T(f))
\end{split}
\end{equation}
for any $x>0$ and $K_T(f)>0$. We now upper bound the right-hand side.

Since $x\mapsto \|x\|^2$ is a smooth map, the function
$$
	\gamma_f(x) = \|\nabla L_b^{-1}[f](x)\|^2 
	- \int_{\T^d}\|\nabla L_b^{-1}[f](y)\|^2 d\mu(y)
$$
is in $\dot{L}^2_\mu \cap H^{d/2+\kappa}$ for all $f\in\mF_T$. Recall the distance $d_L^2(f,g) := \sum_{i=1}^d \| \partial_{x_i}L_b^{-1}[f-g] \|_\infty^2$ defined in Lemma 1 of \cite{nicklray2020}. Using the Sobolev embedding theorem, Lemma \ref{regest} 
and the Runst-Sickel lemma (\cite{nicklray2020}, Lemma 2),
\begin{align*}
	d_L(\gamma_f,0) 
	&\lesssim 
		\left\| \gamma_f \right\|_{H^{d/2+\kappa-1}} \\
	& \lesssim 
		\sum_{i=1}^d \| (\partial_{x_i} L_b^{-1}[f])^2 
		- \| \partial_{x_i} L_b^{-1}[f] \|_{L^2(\mu)}^2 \|_{H^{d/2+\kappa-1}}\\
	& \lesssim 
		\sum_{i=1}^d  \| \partial_{x_i} L_b^{-1}[f]\|_\infty 
		 \| \partial_{x_i} L_b^{-1}[f]\|_{H^{(d/2+\kappa-1)_+}} 
		 + \| \nabla L_b^{-1}[f] \|_{L^2(\mu)}^2 \|1\|_{H^{d/2+\kappa-1}}\\
	& \lesssim \| f\|_{H^{d/2+\kappa-1}}  \| f\|_{H^{(d/2+\kappa-1)_+-1}} 
		+ \|\mu\|_\infty \|f\|_{H^{-1}}^2 \\
	& \lesssim 
		\| f\|_{H^{d/2+\kappa-1}}  \| f\|_{H^{(d/2+\kappa-1)_+-1}} ,
\end{align*}
where the constants depend only on $d,\kappa$ and upper bounds for $\|b\|_{B_{\infty\infty}^{|d/2+\kappa-1|\vee 1}}$ and $\|\mu\|_\infty$. Applying the exponential inequality in Lemma 1 of \cite{nicklray2020} to the class $\{\gamma_f,0\}$ gives for $y\geq 0$,
$$
	P_b \left( |\G_T[\gamma_f]| \geq 2T^{-1/2} \|L_b^{-1}[\gamma_f]\|_\infty 
	+ Cd_L(\gamma_f,0)(1+y) \right) \leq e^{-y^2/2}.
$$
One can identically prove that $\|L_b^{-1}[\gamma_f]\|_\infty \lesssim \| f\|_{H^{d/2+\kappa-1}}  \| f\|_{H^{(d/2+\kappa-1)_+-1}}$ with the constant depending on the same quantities as above. Since $\sqrt{T} \G_T[\gamma_f] = [Z_\cdot(f)]_T -  T \|\nabla L_b^{-1}[f]\|_{L^2(\mu)}^2$, this and the last two displays yield
$$
	P_b \left(  \left| [Z_\cdot(f)]_T  - T \|\nabla L_b^{-1}[f]\|_{L^2(\mu)}^2 \right| 
	\geq C \sqrt{T}  \| f\|_{H^{d/2+\kappa-1}}  \| f\|_{H^{(d/2+\kappa-1)_+-1}}
	(1+y) \right) \leq e^{-y^2/2}
$$
for all $T\geq \eta > 0$. Since by Lemma \ref{regest} in the Supplement, $\|\nabla L_b^{-1}[f]\|_{L^2(\mu)}^2 \lesssim \|\mu\|_\infty \|f\|_{H^{-1}}^2$, substituting the last display into \eqref{QV_split} with $K_T(f) = CT \|f\|_{H^{-1}}^2 + C \sqrt{T}  \| f\|_{H^{d/2+\kappa-1}}$ $  \| f\|_{H^{(d/2+\kappa-1)_+-1}}(1+y)$ gives for all $x,y\geq 0$
\begin{align*}
	P_b&\left( |Z_T(f)| \geq x \right)\\
	&\leq 
	2\exp \left(-\frac{x^2}{2CT \|f\|_{H^{-1}}^2 
	+ 2C \sqrt{T}  \| f\|_{H^{d/2+\kappa-1}}  \| f\|_{H^{(d/2+\kappa-1)_+-1}}
	(1+y) } \right) + e^{-y^2/2}.
\end{align*}

	For $f \in V_J$ and $u\geq -1$, $\|f\|_{H^u}^2 = \sum_{l\leq J} \sum_r 2^{2ul} |\langle f,\Phi_{lr}\rangle_2|^2 \leq 2^{2J(u+1)}\|f\|_{H^{-1}}^2$. Using this bound, that $2^{J[d/2+\kappa + (d/2+\kappa-1)_+]} \lesssim \sqrt{T}$ and setting $y = x/(CT\|f\|_{H^{-1}}^2)^{1/2}$, the last display gives
\begin{align*}
	P_b\left( |Z_T(f)| \geq x \right)  \leq 3\exp \left(-\frac{x^2}{2CT \|f\|_{H^{-1}}^2 
	+ 2C 2^{J[d/2+\kappa+(d/2+\kappa-1)_+]} \| f\|_{H^{-1}} x } \right).
\end{align*}
Using the linearity of $f\mapsto Z_T(f)$ and rearranging, we get the following Bernstein inequality:
\begin{align*}
	P_b\left( |Z_T(f-g)| \geq C\|f-g\|_{H^{-1}} 
	(\sqrt{Tz} + 2^{J[d/2+\kappa+(d/2+\kappa-1)_+]}z) \right)  \leq 3e^{-z},
\end{align*}
where again $C$ depends on $d,\kappa,\eta$ and upper bounds for $\|b\|_{B_{\infty\infty}^{|d/2+\kappa-1|\vee 1}}$ and $\|\mu\|_\infty$.

%


We now apply Theorem 3.5 in Dirksen \cite{D15}, which is a refinement of Theorem 2.2.28 in Talagrand \cite{T14}, to bound the supremum of the random process $(Z_T(f))_{f\in\Fcal_T}$.  In particular, the last display shows that the process has ‘mixed tails' (cf.~(3.8) in \cite{D15}) with respect to the metrics $d_1(f,g) = C2^{J[d/2+\kappa+(d/2+\kappa-1)_+]}\|f-g\|_{H^{-1}}$ and $d_2(f,g) = C\sqrt T\|f-g\|_{H^{-1}}$. The diameters $\Delta_{d_1}(\Fcal_T)$ and $\Delta_{d_2}(\Fcal_T)$ appearing in the second display of Theorem 3.5 in \cite{D15} can be bounded by
$$
	\Delta_{d_1}(\Fcal_T):=\sup_{f,g\in\Fcal_T}d_1(f,g)\le 2C2^{J[d/2+\kappa+(d/2+\kappa-1)_+]}|\Fcal_T|_{H^{-1}},
$$ 
$$
	\Delta_{d_2}(\Fcal_T):=\sup_{f,g\in\Fcal_T}d_2(f,g)
	\le 2C\sqrt{T}|\Fcal_T|_{H^{-1}},
$$ 
so that Theorem 3.5 of \cite{D15} yields that for all $x\geq 1$,
$$P_b\left( \sup_{f\in \mF_T} |Z_T(f)| \geq C \left( \gamma_2(\mF_T,d_2) + \gamma_1 (\mF_T,d_1) + |\mF_T|_{H^{-1}} \left\{ \sqrt{Tx} + 2^{J[d/2+\kappa+(d/2+\kappa-1)_+]} x \right\} \right) \right)\leq e^{-x},$$
where $\gamma_1, \gamma_2$ are `generic chaining functionals'. Using the estimate (2.3) in \cite{D15}, and recalling that $\Fcal_T\subset V_J$ has dimension $D_{\Fcal_T},$ we can bound 
$\gamma_2(\Fcal_T,d_2)$ by a multiple of
\begin{align*}
	\int_0^\infty \sqrt{\log N(\Fcal_T, d_2,\eta)}d\eta
	&\le
		\int_0^{\Delta_{d_2}(\Fcal_T)}
		\sqrt{D_{\mF_T}\log\left( \frac{3\Delta_{d_2}(\Fcal_T)}
		{\eta}\right)}d\eta\\
	& = D_{\mF_T}^{1/2} \Delta_{d_2}(\Fcal_T) \int_0^1
		\sqrt{\log\left( 3/u\right)}du\\
	&\le
		C_1D_{\mF_T}^{1/2}\sqrt{T}|\Fcal_T|_{H^{-1}}
\end{align*}
for $C_1>0$, where the first inequality follows from the usual metric entropy estimate for balls in finite-dimensional Euclidean spaces \cite[Proposition 4.3.34]{ginenickl2016}.
Similarly, $\gamma_1(\Fcal_T,d_1)$ is bounded by a multiple of
\begin{align*}
	\int_0^\infty \log N(\Fcal_T, d_1,\eta) d\eta
	&\le
		D_{\mF_T}\int_0^{\Delta_{d_1}(\Fcal_T)}\log\left( \frac{3\Delta_{d_1}
		(\Fcal_T)}{\eta}\right)d\eta\\
	&\le
		C_2D_{\mF_T}\Delta_{d_1}(\Fcal_T)
	\le
		C_3D_{\mF_T}2^{J[d/2+\kappa+(d/2+\kappa-1)_+]}
		|\Fcal_T|_{H^{-1}}.
\end{align*}
In summary, Theorem 3.5 of \cite{D15} implies that for all $x\geq 1$
$$P_b\left( \sup_{f\in \mF_T} |Z_T(f)| \geq C|\mF_T|_{H^{-1}} \left\{ \sqrt{T}(D_{\mF_T}^{1/2}+\sqrt{x}) + 2^{J[d/2+\kappa+(d/2+\kappa-1)_+]}(D_{\mF_T}+ x) \right\} \right)\leq e^{-x}.$$
This provides the concentration inequality for the first term in \eqref{subgauss_process}. For the second term in \eqref{subgauss_process}, using the Sobolev embedding theorem and Lemma \ref{regest} in the Supplement,
$$
	\|L_b^{-1}[f]\|_\infty \lesssim \|L_b^{-1}[f]\|_{H^{d/2+\kappa}} 
	\lesssim
	 \|f\|_{H^{d/2+\kappa-1}} \leq 2^{J[d/2+\kappa]}\|f\|_{H^{-1}} 
	\lesssim 2^{J[d/2+\kappa+(d/2+\kappa-1)_+]}|\Fcal_T|_{H^{-1}}
$$
for all $f\in \mF_T$, where the constant depends on $d,\kappa,\eta$ and an upper bound for $\|b\|_{B_{\infty\infty}^{|d/2+\kappa-1|\vee 1}}$. The result then follows from the last two displays and \eqref{subgauss_process}.
\endproof 

\end{proposition}

\subsection{Construction of tests and proof of Lemma \ref{lem:tests}}\label{sec:tests}

Using Proposition \ref{prop:sup_prob}, we now derive concentration inequalities for a preliminary estimator $\hat{\mu}_T$ of the invariant measure $\mu$, uniformly over certain sets $\Gamma(\delta,m)$ below. We can then exploit the correspondence $\mu \propto e^{2B}$ to obtain an estimator for the gradient vector field $\nabla B = \tfrac{1}{2} \nabla \log \mu = \tfrac{1}{2} (\nabla \mu)/\mu$ based on $\hat{\mu}_T$. Consider the wavelet estimator for the invariant measure
\begin{equation}\label{wav_estimator}
	\hat{\mu}_T(x) = \sum_{l \leq J} \sum_r \hat{\beta}_{lr} \Phi_{lr}(x),
	\qquad x\in\T^d,
\end{equation}
where $\hat{\beta}_{lr} = \tfrac{1}{T} \int_0^T \Phi_{lr}(X_t) dt$ and $J \to\infty$ as $T\to\infty$. We now prove the required concentration inequality for the induced estimator $\tfrac{1}{2} \nabla \log\hat{\mu}_T$ of $\nabla B$ under the conditions of Theorem \ref{thm:contraction_b}.
Recall that $P_J$ denotes the $L^2$-projection onto the wavelet approximation space $V_J$ given in \eqref{V_J}, and that we write $f\in \dot{L}^2_{\mu}$ if $\int_{\T^d} f d\mu = 0$.

\begin{lemma}\label{lem:W11_exp}
Let $q \in[1,2]$, $2^J\to \infty$ and $\eps_T\to 0$ satisfy $T\eps_T^2 \to \infty$. Assume further that 
$$
	2^{J[d/2+\kappa+(d/2+\kappa-1)_+]}\eps_T=O(1) \qquad \text{and} \qquad 
	T^{-1/2} 2^{J[d+\kappa+(d/2+\kappa-1)_+]}=O(1)
$$
for some $\kappa>0$. Define
\begin{align*}
	\Gamma = \Gamma(\delta,m) =  \left\{ \mu  : \int_{\T^d} \mu(x)dx = 1,
	 ~ \mu(x) \geq \delta \text{ for all }x\in \T^d , 
	 ~ \|\mu\|_{C^{(d/2+\kappa)\vee 2}} \leq m \right\}
\end{align*}
for some $\delta,m>0$. Then there exists $C = C(\delta,m,q)>0$ such that for all $T>0$ large enough and all $M>0$,
\begin{align*}
	 \sup_{\substack{B: \nabla B = \tfrac{1}{2} \nabla \log  \mu,\\ \mu\in \Gamma}}
	 & P_B \left(  \|\nabla 
	 \log \hat{\mu}_T - \nabla \log \mu\|_q \geq C\|\mu - P_J\mu\|_{W^{1,q}} 
	 + C(1+M)  \left( T^{-1/2} 2^{Jd/2} + \eps_T \right) \right) \\
	 & \qquad \leq 4e^{-MT\eps_T^2}.
\end{align*}

\begin{proof}
All constants in this proof are taken uniform over $\delta,m>0$ in the definition of the set $\Gamma$ above and we write $b_\mu = \frac{1}{2}\nabla \log \mu$. We first derive a preliminary $L^\infty$-convergence rate $\zeta_T$ for $\hat{\mu}_T$, for which we establish an exponential inequality. Write
\begin{align*}
	\|\hat{\mu}_T-P_J\mu\|_\infty 
	& = 
		\sup_{x\in \T^d} \left| \frac{1}{T} \int_0^T \sum_{l\leq J} 
		\sum_r [\Phi_{lr}(X_s)-\langle \mu, \Phi_{lr} \rangle_2 ]\Phi_{lr}(x) ds \right| 
	= 
		\frac{1}{\sqrt{T}} \sup_{x\in \T^d} |\G_T[h_x]|
\end{align*}
with $h_x(u) = \sum_{l\leq J} \sum_r [\Phi_{lr}(u) - \langle \mu,\Phi_{lr} \rangle_2] \Phi_{lr}(x) \in V_J \cap \dot{L}^2_\mu$ and $\G_T[f]=\tfrac{1}{\sqrt{T}}\int_0^Tf(X_s)ds$ the empirical process for $\mu$-centered functions $f\in \dot{L}^2_\mu$. Recalling that the periodized father wavelet $\Phi_{-10} \equiv 1$ (\cite{ginenickl2016}, p. 354) and that $\sup_x \sum_r \Phi_{lr}(x)^2 \lesssim 2^{Jd}$,
\begin{align*}
	\|h_x\|_{H^{-1}}^2 
	&= 
		\sum_{l=0}^J \sum_r 2^{-2l} |\Phi_{lr}(x)|^2 + \left| \sum_{l\leq J} 
		\sum_r \langle \mu,\Phi_{lr} \rangle_2 \Phi_{lr}(x)\right|^2 \\
	& \lesssim 
		\sum_{l=0}^J 2^{J(d-2)} + \|P_J\mu\|_\infty^2
	\lesssim 
		J_d 2^{J(d-2)_+},
\end{align*}
where $J_d = J^{1_{\{d=2\}}}$ and we have used $\|P_J\mu\|_\infty \leq \|\mu\|_\infty + \|\mu-P_J\mu\|_\infty \lesssim (1+2^{-J})\|\mu\|_{B_{\infty\infty}^1}$ $\lesssim m$. Applying Proposition \ref{prop:sup_prob} with $D_{\mF_T}= \text{dim}(\T^d)=d$, $|\mF_T|_{H^{-1}}=\sup_{x\in\T^d} \|h_x\|_{H^{-1}} = O(J_d^{1/2}$ $ 2^{J(d/2-1)_+})$ and $x=MT\eps_T^2\to\infty$ then gives
$$
	P_b\left( \|\hat{\mu}_T-P_J\mu\|_\infty \geq C(1+M)J_d^{1/2} 2^{J(d/2-1)_+}
	 \eps_T \right)\leq 2e^{-MT\eps_T^2}
$$
for $T>0$ large enough, having used that $2^{J[d/2+\kappa+(d/2+\kappa-1)_+]}\eps_T\lesssim 1$. Note that the constant in the last display depends on $d$, $\kappa$ and upper bounds for $\|\mu\|_\infty\lesssim \|\mu\|_{B_{\infty\infty}^1} \lesssim m$ and $\|b_\mu\|_{B_{\infty\infty}^{|d/2+\kappa-1|\vee 1}}$. For this last quantity, by the chain rule, for $\mu \geq \delta$ bounded away from zero, $\|b_\mu\|_{B_{\infty\infty}^s} \lesssim \|\nabla \log \mu\|_{B_{\infty\infty}^s} \lesssim \|\log \mu\|_{B_{\infty\infty}^{s+1}} \lesssim \|\log \mu\|_{C^{s+1}} \lesssim 1+\|\mu\|_{C^{s+1}}$ for all $s>0$. Thus in particular, $\|b_\mu\|_{B_{\infty\infty}^{|d/2+\kappa-1|\vee 1}} \lesssim 1+\|\mu\|_{C^{(d/2+\kappa)\vee 2}}\lesssim 1+m$ is also uniformly bounded over $\Gamma$. Combined with the bias bound $\|\mu-P_J\mu\|_\infty \lesssim 2^{-J} \|\mu\|_{B_{\infty\infty}^1} \lesssim 2^{-J}m$, this yields
\begin{equation}\label{L_inf_ineq}
P_b\left( \|\hat{\mu}_T-\mu\|_\infty \geq C(1+M)J_d^{1/2} 2^{J(d/2-1)_+} \eps_T + C2^{-J}m \right)\leq 2e^{-MT\eps_T^2}
\end{equation}
with $C$ a uniform constant over $\Gamma$. Set $\zeta_T = C(1+M)J_d^{1/2} 2^{J(d/2-1)_+} \eps_T + C2^{-J}m \lesssim 2^{J[d/2+\kappa+(d/2+\kappa-1)_+]}\eps_T + 2^{-J}m \to 0$. On the event $\{ \|\hat{\mu}_T - \mu\|_\infty \leq \zeta_T\}$, for $x\in \T^d$, $\mu \in \Gamma$ and $i=1,\dots,d$,
\begin{align*}
|\partial_{x_i} \log \hat{\mu}_T(x) - \partial_{x_i} \log \mu(x)| & = \left| \frac{\partial_{x_i} \hat{\mu}_T(x)}{\hat{\mu}_T(x)} - \frac{\partial_{x_i} \mu(x)}{\mu(x)} \right| \\
& \leq \left| \frac{\partial_{x_i} \hat{\mu}_T(x) - \partial_{x_i} \mu(x)}{\hat{\mu}_T(x)}\right| +|\partial_{x_i} \mu(x)| \left| \frac{\mu(x) - \hat{\mu}_T(x)}{\hat{\mu}_T(x) \mu(x)} \right| \\
& \leq  \frac{|\partial_{x_i} \hat{\mu}_T(x) - \partial_{x_i} \mu(x)|}{\delta - \zeta_T} +|\partial_{x_i} \mu(x)| \frac{|\hat{\mu}_T(x) - \mu(x)|}{\delta(\delta-\zeta_T)}.
\end{align*}
Taking the $q^{th}$ power, integrating, using that $\|\partial_{x_i} \mu\|_\infty \lesssim \|\mu\|_{C^2} \leq m$ and \eqref{L_inf_ineq} gives for large enough $T>0$,
\begin{equation}\label{log_bound}
\inf_{\substack{B: \nabla B = \tfrac{1}{2} \nabla \log  \mu,\\ \mu\in \Gamma}} P_B \left( \|\nabla \log \hat{\mu}_T - \nabla \log \mu\|_q \leq C(\delta,m) \|\hat{\mu}_T-\mu\|_{W^{1,q}} \right) \geq 1- 2e^{-MT\eps_T^2}
\end{equation}
since $\zeta_T \to 0$. It thus suffices to prove an exponential inequality for $\|\hat{\mu}_T-\mu\|_{W^{1,q}}$.

For $1\leq q \leq 2$, we have the continuous embedding $H^1(\T^d) = W^{1,2}(\T^d) \subset W^{1,q}(\T^d)$. For the variance term, by Hilbert space duality,
\begin{align*}
\|\hat{\mu}_T - P_J\mu\|_{H^1} &= \sup_{\varphi \in C^\infty: \|\varphi\|_{H^{-1}} \leq 1} \left| \int_{\T^d} [\hat{\mu}_T(x) - P_J\mu(x)] \varphi(x) dx \right| \\
& = \sup_{\varphi \in V_J: \|\varphi\|_{H^{-1}} \leq 1} \left| \sum_{l\leq J} \sum_r \langle \hat{\mu}_T - \mu,\Phi_{lr} \rangle_2  \langle \varphi, \Phi_{lr} \rangle_2 \right| \\
& = \sup_{\varphi \in V_J: \|\varphi\|_{H^{-1}} \leq 1} \left| \frac{1}{T}\int_0^T \sum_{l\leq J} \sum_r [\Phi_{lr}(X_s) - \langle \mu,\Phi_{lr} \rangle_2]  \langle \varphi, \Phi_{lr} \rangle_2 ds \right| \\
& = \frac{1}{\sqrt{T}} \sup_{\varphi \in V_J: \|\varphi\|_{H^{-1}} \leq 1} \left| \G_T[g_\varphi]  \right|,
\end{align*}
where $g_\varphi(u) = \sum_{l\leq J} \sum_r [\Phi_{lr}(u) - \langle \mu,\Phi_{lr} \rangle_2]  \langle \varphi, \Phi_{lr} \rangle_2 \in V_J \cap \dot{L}^2_\mu.$ Using that $\Phi_{-10} \equiv 1$, for $\|\varphi\|_{H^{-1}} \leq 1$,
\begin{align*}
	\|g_\varphi\|_{H^{-1}}^2 
	& =
		\sum_{l=0}^J \sum_r 2^{-2l}  |\langle \varphi , \Phi_{lr}\rangle_2|^2 
		+ \left| \sum_{l \leq J} \sum_r \langle \mu,\Phi_{lr} \rangle_2 \langle
		\varphi, \Phi_{lr} \rangle_2 \right|^2  \\
	& \leq 
		\|\varphi\|_{H^{-1}}^2 + |\langle \mu, \varphi \rangle_2|^2  \\ 
	& \leq 
		(1 + \|\mu\|_{H^1}^2 )\|\varphi\|_{H^{-1}}^2 \\
	&\lesssim 
		1 + \|\mu\|_{C^2}^2
	\leq 
		1 + m^2.
\end{align*}
Applying now Proposition \ref{prop:sup_prob} with $D_{\mF_T}= \text{dim}(V_J)=O(2^{Jd})$, 
$|\mF_T|_{H^{-1}}=\sup_{\varphi} \|g_\varphi \|_{H^{-1}} \lesssim 1+m$
 and $x=MT\eps_T^2\to\infty$ gives
$$P_B\left( \|\hat{\mu}_T - P_J\mu\|_{H^1} \geq C(1+M)  \left( T^{-1/2} 2^{Jd/2} + \eps_T \right) \right)\leq 2e^{-MT\eps_T^2},$$
where we have used $2^{J[d/2+\kappa+(d/2+\kappa-1)_+]}\eps_T\lesssim 1$ and $T^{-1/2} 2^{J[d+\kappa+(d/2+\kappa-1)_+]}\lesssim 1$ and where the constant in the last display again depends on $d$, $\kappa$ and $m$.
Using the embedding $H^1(\T^d) \subset W^{1,q}(\T^d)$, $1\leq q \leq 2$, this yields
$$P_B\left( \|\hat{\mu}_T - \mu\|_{W^{1,q}} \geq \|\mu - P_J\mu\|_{W^{1,q}} + C(1+M)  \left( T^{-1/2} 2^{Jd/2} + \eps_T \right) \right)\leq 2e^{-MT\eps_T^2}.$$
Combining the last inequality with \eqref{log_bound} proves the result.
\end{proof}
\end{lemma}

\proof[Proof of Lemma \ref{lem:tests}]
	Consider the test $\phi_T  = 1\{ \|\nabla \log \hat{\mu}_T - \nabla \log \mu_0\|_q \geq M_0 \xi_T\}$, where $\hat\mu_T$ is the wavelet estimator in \eqref{wav_estimator} and $M_0$ is to be selected below. Since $T^{-1/2} 2^{Jd/2} + \eps_T \lesssim \xi_T$ and $\|P_J \mu_0-\mu_0\|_{W^{1,q}} \leq D_0\xi_T$ by assumption, Lemma \ref{lem:W11_exp} gives that for any $M>0$ and large enough $T>0$,
\begin{align*}
P_{B_0} \left( \|\nabla \log \hat{\mu}_T - \nabla \log \mu_0\|_q \geq  C(D_0 + 1+M) \xi_T \right) \leq 4e^{-MT\eps_T^2}.
\end{align*}
Taking $M_0 > C(2+D_0)$, the type-I error then satisfies $P_{B_0} \phi_T \leq 4e^{-T\eps_T^2} \to 0$.

Turning to the type-II error, since $\|\mu-P_J\mu\|_{W^{1,q}} \leq C_\Lambda \xi_T$ for all $\mu \in \Lambda_T$, and $\Lambda_T \subset \Gamma(\delta,m)$ for $\Gamma(\delta,m)$ the set in Lemma \ref{lem:W11_exp}, applying that lemma yields that for all $M>0$ and large enough $T>0$,
\begin{align*}
 \sup_{\substack{B: \nabla B = \tfrac{1}{2} \nabla \log  \mu,\\ \mu\in \Gamma}} P_B\left( \|\nabla \log \hat{\mu}_T - \nabla \log \mu\|_q \geq C(C_\Lambda + 1+M) \xi_T \right) \leq 4e^{-MT\eps_T^2}.
\end{align*}
Now consider $B \in \Theta_T$ (implying $\nabla B = \tfrac{1}{2}\nabla \log \mu$ for some $\mu \in \Lambda_T$) such that $\|\nabla B - \nabla B_0\|_q = \tfrac{1}{2} \|\nabla \log\mu - \nabla \log \mu_0\|_q \geq D\xi_T$. Applying the triangle inequality and the last display,
\begin{align*}
P_B(1-\phi_T) & = P_B (\|\nabla \log \hat{\mu}_T - \nabla \log \mu_0\|_q \leq M_0 \xi_T)\\
& \leq P_B (\|\nabla \log \mu_0- \nabla \log \mu\|_q  - \|\nabla \log \mu - \nabla \log \hat{\mu}_T \|_q \leq M_0 \xi_T) \\
& \leq P_B ( (2D-M_0) \xi_T \leq \|\nabla \log \mu - \nabla \log \hat{\mu}_T \|_q )
\leq 4e^{-MT\eps_T^2},
\end{align*}
for $D>M_0/2 + C(C_\Lambda+1+M)/2$. This completes the proof.
\endproof 

\section{Proofs for Gaussian and $p$-exponential priors}
\label{Sec:ProofsGauss}

\subsection{Proof of Theorem \ref{Prop:GaussRates}}
\label{Sec:GaussRatesProof}

We verify the assumptions of Theorem \ref{thm:contraction_b} with $q=2$, $\eps_T = \xi_T \simeq T^{-s/(2s+d)}$ and $2^J \simeq T^{1/(2s+d)}$. The quantitative conditions \eqref{Eq:QuantCond} in Theorem \ref{thm:contraction_b} are then satisfied for all \textcolor{blue}{$s>(d-1)\vee(1/2)$} and all $0 < \kappa \leq (s+1-d)/2$ if $d\geq 2$ and $0 < \kappa \leq (s-1/2) \wedge (1/2)$ if $d=1$. In particular, we can take $\kappa >0$ arbitrarily small in what follows.

	As pointed out in Remark \ref{rem:SB}, the ‘small ball condition' (assumption  (ii)  in Theorem \ref{thm:contraction_b}) follows if we show that for some $C>0$,
\begin{equation}
\label{Eq:NthSmallBall}
	\Pi\left(\sum_{i=1}^d\|\partial_{x_i} B - \partial_{x_i} B_0\|^{2}_\infty 
	\le \eps_T^2\right )\ge e^{-CT\eps_T^2}.
\end{equation}
Note that
\begin{equation*}
\label{C1_bound}
	\left( \sum_{i=1}^d \|\partial_{x_i} B - \partial_{x_i} B_{0}\|_\infty^2 \right)^{1/2}
	\leq \sum_{i=1}^d \|\partial_{x_i}B - \partial_{x_i}B_0\|_\infty 
	\lesssim \|B-B_0\|_{C^1},
\end{equation*}
and since $\|B_0-B_{0,T}\|_{C^1}=O(T^{-s/(2s+d)})=O(\eps_T)$ by assumption, it thus suffices to lower bound $\Pi (\|B-B_{0,T}\|_{C^1} \leq \eps_T/2)$, upon replacing $\eps_T$ with a multiple of itself if necessary. Recall that the RKHS $\H_B$  of the scaled Gaussian process $B=W/T^{d/(4s+2d)}=W/(\sqrt{T}\eps_T)$ equals the RKHS $\H$ of $W$, with scaled norm $\|h\|_{\H_B}=\sqrt{T}\eps_T\|h\|_\H$. Since $\|B_{0,T}\|_\H = O(1)$, using Corollary 2.6.18 of \cite{ginenickl2016} we lower bound the probability of interest by
$$
	e^{-\frac{1}{2}\|B_{0,T}\|^2_{\H_B}} \Pi(\|B\|_{C^1} \leq \eps_T/2)
	\geq e^{-c_1T\eps_T^2} \Pi_W(\|W\|_{C^1} \leq \sqrt{T}\eps_T^2/2)
$$
for some $c_1>0$. Since $\sqrt{T}\eps_T^2 \to 0$ for $s>d/2$, the small ball estimate \eqref{SB_C1} in the Supplement \cite{supp} gives 
$$
	\Pi_W(\|W\|_{C^1} \leq \sqrt{T}\eps_T^2/2) 
	\geq e^{-c_2(\sqrt{T}\eps_T^2)^{-2d/(2s-d)}}
	= e^{-c_3 T\eps_T^2},
$$
for some $c_2,c_3>0$. Taking $C=c_1+c_3<\infty$, the last two displays yields \eqref{Eq:NthSmallBall} as required.

	For $M>0$ and $p=2$, let $\mB_T$ be the set in \eqref{B_T set} and define $\Lambda_T = \{ \mu_B = e^{2B}/\int_{\T^d}e^{2B}: B\in \mB_T\}$. Taking $M>0$ large enough, Lemma \ref{lem:prior_prob_event} (i) implies that $\Pi(\mB_T^c) \leq e^{-(C+4)T\eps_T^2}$ as required by assumption (i) in Theorem \ref{thm:contraction_b}. We now show that $\Lambda_T$ satisfies the inclusion condition \eqref{Eq:LambdaTProp}. Since $\|B\|_\infty \leq \|B\|_{C^{(d/2+\kappa)\vee 2}} \leq M$ for every $B\in \mB_T$, it holds that $\mu_B \geq e^{-4M}>0$ for all $\mu_B\in \Lambda_T$. Moreover, using again the boundedness of $B\in \mB_T$,
$$
	\|\mu_B\|_{C^{(d/2+\kappa)\vee 2}} 
	\leq e^{2M} \|e^{2B}\|_{C^{(d/2+\kappa)\vee 2}} 
	\lesssim 1 + \|B\|_{C^{(d/2+\kappa)\vee 2}} 
	+ \|B\|_{C^{(d/2+\kappa)\vee 2}}^{(d/2+\kappa)\vee 2} 
	\leq c(M),
$$ 
where for $d\le3$ the second inequality (with $C^{(d/2+\kappa)\vee 2}=C^2$ for $\kappa$ small enough) follows readily by differentiation, while for $d\ge 4$ it is implied by Lemma \ref{lem:exp_map} upon noting that $C^{(d/2+\kappa)\vee 2}=C^{d/2+\kappa}=B^{d/2+\kappa}_{\infty\infty}$, since $d/2+\kappa\notin\N$ for $\kappa$ small enough (cf.~Chapter 3 in \cite{ST87}). Finally, the bias bound follows from Lemma \ref{lem:bias} with $p=2$. This shows that $\Lambda_T$ satisfies the required assumptions in Theorem \ref{thm:contraction_b}.

	It remains only to consider the true potential $B_0$. However, since $B_0 \in H^{s+1}$, Lemma \ref{lem:exp_map} similarly implies that $\|\mu_0-P_J\mu_0\|_{W^{1,2}} \lesssim 2^{-Js} \|\mu_0\|_{H^{s+1}} \simeq \eps_T.$ The result thus follows from Theorem \ref{thm:contraction_b}.
\qed

\subsection{Proof of Theorem \ref{Prop:PExpRates}}
\label{Sec:pExpRatesProof}

We verify the assumptions of Theorem \ref{thm:contraction_b} with $q=p$, $\eps_T 
= \xi_T  \simeq T^{-s/(2s+d)}$ and $2^J \simeq T^{1/(2s+d)}$. Since by assumption $s>d/p+(d/2)\vee2 > (d-1)\vee (1/2)$, the quantitative conditions \eqref{Eq:QuantCond} in Theorem \ref{thm:contraction_b} are again satisfied for all $0 < \kappa \leq (s+1-d)/2$ if $d\geq 2$ and $0 < \kappa \leq (s-1/2) \wedge (1/2)$ if $d=1$. In particular, we can take $\kappa >0$ arbitrarily small in what follows. The remaining assumptions are verified using tools for $p$-exponential measures mainly due to \cite{ADH20}.

	We first consider assumption (ii)  in Theorem \ref{thm:contraction_b} with $\mathcal{SB}_T = \{B \in V_J \cap \dot{L}^2: \|\nabla B - \nabla B_0\|_{L^2(\mu_0)} \leq M\eps_T\}$ as suggested in Remark \ref{rem:SB}. Since $s>d/p+(d/2)\vee2 \geq d\vee (3/2)$ for $p\in[1,2]$ and the prior $\Pi$ arising as the law of $B$ in \eqref{Eq:pExpPrior} is supported on $V_J\cap \dot{L}^2$, Lemma \ref{Lemma:AnotherSmallBall} below implies that it is enough to show that for some $M,C>0$
$$
	\Pi( \|\nabla B - \nabla B_0\|_{L^2(\mu_0)} \le M\eps_T)\ge 
	e^{-CT\eps_T^2},
$$
where $\mu_0 \propto e^{2B_0}$ is the invariant density and $\|f\|^2_{L^2(\mu_0)} = \int_{\T^d}|f(x)|^2\mu_0(x)dx$. Since $\mu_0$ is bounded (\cite{nicklray2020}, Proposition 1), the probability on the left hand side is greater than 
$$
	\Pi(  \|\nabla B - \nabla B_0\|_{2} \le m_1\eps_T)\ge 
	\Pi(  \|B -  B_0\|_{H^1} \le m_2\eps_T)
$$	
for some $m_1, m_2 >0$. Let $P_JB_0$ the wavelet projection of $B_0\in H^{s+1}\cap \dot{L^2}$ onto $V_J$. Then $\|B_0 - P_J B_0\|_{H^1}\lesssim 2^{-J s}\simeq \eps_T$, so that by the triangle inequality the latter probability is lower bounded by $\Pi( \|B -  P_JB_0\|_{H^1} \le m_3\eps_T)$ for some $m_3>0$. In the language of \cite{ADH20}, the $\Zcal$-space (Definition 2.8 in \cite{ADH20}) associated to the $p$-exponential random element $B=W/(T^{d/(2s+d)})^{1/p}=W/(T\eps_T^2)^{1/p}$ is equal to $V_J\cap \dot{L^2}$, with norm
\begin{equation}
\label{Eq:Znorm}
	\|h\|_\Zcal
	=
		\Bigg(T\eps_T^2\sum_{l =0}^J \sum_r
		2^{p l \big(s+1+\frac{d}{2}-\frac{d}{p}\big)}|\langle h,\Phi_{l r}
		\rangle_2|^p\Bigg)^\frac{1}{p}
	=
		\big(T\eps_T^2\big)^\frac{1}{p}\|h\|_{B^{s+1}_{pp}},
	\qquad h\in \Zcal.
\end{equation}
Since $P_JB_0\in \Zcal$, by Proposition 2.11 in \cite{ADH20} the probability of interest is thus greater than
\begin{align*}
	e^{-\frac{1}{p} \|P_{J}B_0\|^p_{\Zcal}}
	\Pi\big(\|B\|_{H^1}\le m_3\eps_T\big)
	&\ge
		e^{-\frac{1}{p} T\eps_T^2 \|B_0\|^p_{B^{s+1}_{pp}}}
		\Pi\big(\|B\|_{H^1}\le m_3\eps_T\big),
\end{align*}
where $\|B_0\|_{B^{s+1}_{pp}}<\infty$ in view of the continuous embedding $H^{s+1}(\T^d)\subseteq B^{s+1}_{pp}(\T^d)$ holding for all $p\le2$ (p. 33 in \cite{LSS09}). We conclude by estimating the above centred small ball probability. Using Theorem 4.2 in \cite{A07} 
(whose conclusion can readily be adapted for double index sums), we have as $T\to\infty$,
\begin{align*}
	-\log \Pi\big(\|B\|_{H^1}\le m_3\eps_T\big) 
	&= -\log \Pi_W\big(\|W\|_{H^1} \le m_3 (T\eps_T^2)^\frac{1}{p}\eps_T\big)\\
	&\simeq \left[(T\eps_T^2)^\frac{1}{p}\eps_T\right]^{-\frac{d}{s-d/p}}
	= T\eps_T^2.
\end{align*}
Thus, for $c_1>0$ a large enough constant and $C=c_1+\|B_0\|_{B^{s+1}_{pp}}^p/p<\infty$, we obtain as required that $\Pi\left( \|\nabla B-\nabla B_0\|_{L^2(\mu_0)} \le M \eps_T\right) \ge e^{-CT\eps_T^2}$.

	The remaining conditions in Theorem \ref{thm:contraction_b} are verified arguing as in the proof of Theorem \ref{Prop:GaussRates}, using the sets $\mB_T$ in \eqref{B_T set} with $M>0$ large enough and $p\in[1,2]$, taking $\Lambda_T = \{ \mu_B = e^{2B}/\int_{\T^d}e^{2B}: B\in \mB_T\}$, and noting that since $B_0 \in B_{pp}^{s+1}$, we have by Lemma \ref{lem:exp_map} that $\|\mu_0-P_J\mu_0\|_{W^{1,p}} \lesssim 2^{-Js} \|\mu_0\|_{B_{pp}^{s+1}} \lesssim \eps_T$.
\qed

\begin{lemma}\label{Lemma:AnotherSmallBall}
Suppose $B_0 \in H^{s+1}(\T^d) \cap \dot{L}^2(\T^d) $ for $s > d \vee (3/2)$. If $\eps_T = T^{-s/(2s+d)}$, $J \in \mathbb{N}$ satisfies $2^J \simeq T^{1/(2s+d)}$ and $M>0$, then as $T\to\infty$,
\begin{align*}
	P_{B_0} & \left(  \sup_{B \in V_J \cap \dot{L}^2: \|\nabla B - \nabla B_0\|
	_{L^2(\mu_0)} \leq M\eps_T} \frac{1}{T} \int_0^T  
	\|\nabla B(X_s)-\nabla B_0(X_s)\|^2 ds \leq M^2 \eps_T^2 
	+ o(\eps_T^2) \right) \\
	& \qquad \qquad \to 1.
\end{align*}
\end{lemma}

The proof is deferred to Section \ref{ProofOfProof} of the Supplement. The above restriction to $\dot{L}^2$ is for identifiability and simplifies certain norms. A similar result holds under other identifiability constraints.

\subsection{Concentration inequalities and bias bounds}

In this section, we build on techniques for Gaussian process priors \cite{vdvaart2008}, $p$-exponential priors \cite{ADH20} and rescaled priors \cite{MNP20} to obtain the prior bias bounds needed to apply our general contraction theorem for multi-dimensional diffusions. The proofs of the following lemmas are deferred to Section \ref{ProofOfProof} in the Supplement.

\begin{lemma}\label{lem:prior_prob_event}
	For $s,M,\kappa>0$, $p\in[1,2]$ and sequences  $\eps_T=$ $ T^{-s/(2s+d)},\ \overline\eps_T=T^{-(s+1)/(2s+d)}$, define the sets
\begin{equation}\label{B_T set}
	\mB_T 
	=
		\{ B = B_1 + B_2: \|B_1\|_\infty \leq \overline{\eps}_T, 
		\|B_1\|_{C^1} \leq \eps_T,
		\|B_2\|_{B^{s+1}_{pp}} \leq M, 
		\|B\|_{C^{(d/2+\kappa)\vee 2}} \leq M \}.
\end{equation}
Assume either:
\begin{enumerate}
\item[(i)] $p=2$ and $B = W/(\sqrt{T}\eps_T)$ for $W\sim \Pi_W$ a Gaussian process satisfying Condition \ref{GP_condition};

\item[(ii)] $B = W/(T\eps_T^2)^\frac{1}{p}$ for $W\sim\Pi_W$ a $p$-exponential random element as in \eqref{Eq:pExpBasePrior} with $s>(d/2+\kappa)\vee 2+d/p-1$.
\end{enumerate}
Let $\Pi=\Pi_T$ be the law of $B$. Then, for every $K>0$, there exists $M>0$ large enough such that $\Pi(\mB_T^c) \leq e^{-KT\eps_T^2}.$

\end{lemma}

\begin{lemma}\label{lem:exp_map}
Let $1 \leq p,q\leq \infty$ and $t>d/p$. If $\|B\|_\infty \leq m$, then 
$$
	\|e^B\|_{B_{pq}^t} \leq C(1+\|B\|_{B_{pq}^t} + \|B\|_{B_{pq}^t}^t)
$$
for some constant $C=C(m,t,p,q)>0$.

\proof

Consider a function $f \in C^\infty(\R)$ such that $f(x) = e^x-1$ for $|x|\leq m$ and $\|f\|_{L^\infty(\R)} \leq 2m$. Since $\|B\|_\infty \leq m$, we have $f\circ B(x) = e^{B(x)}-1$ for all $x\in \T^d$. By Theorem 11 of Bourdaud and Sickel \cite{bourdaud2010},
$$\|e^B-1\|_{B_{pq}^t} \leq c \|f\|_{C_b^{\lfloor t \rfloor+1}(\R)} (\|B\|_{B_{pq}^t} + \|B\|_{B_{pq}^t}^t)$$
for some $c>0$. Since $\|1\|_{B_{pq}^t}<\infty$ for periodic Besov spaces, the result follows.
\endproof

\end{lemma}

\begin{lemma}\label{lem:bias}
For $s,M,\kappa>0$ and $p\in[1,2]$, let $\mB_T$ be the set in \eqref{B_T set}, with $\eps_T, \overline \eps_T$ as in Lemma \ref{lem:prior_prob_event}. If $2^J \simeq T^{1/(2s+d)}$, then there exists a finite constant $C$ depending on $s,d,m$ and the wavelet basis $\{\Phi_{l r}\}$ such that
$$ 
	\left\{ \mu_B = \frac{e^{2B}}{\int_{\T^d} e^{2B(x)}dx}: B \in \mB_T \right\} 
	\subset \{ \mu: \|\mu - P_J\mu\|_{W^{1,p}} \leq C\eps_T\}.
$$

\end{lemma}

\section*{Acknowledgements} 
We would like to thank Andrew Stuart for raising the question that led to this research, Richard Nickl for valuable discussions, and the AE and three referees for many helpful comments that improved the manuscript. M.G.~was supported by the European Research Council under ERC grant agreement No.647812 (UQMSI), and during part of the revision of the manuscript was affiliated with the University of Oxford and supported by the ERC grant agreement No.834275 (GTBB).

\begin{supplement}
\stitle{}
In this supplement, we provide additional background material and all the technical results and proofs not included in the main article. For the convenience of the reader, we repeat here the statements of results in the main article whose proofs appear in this supplement. We linearly continue the equation numbering scheme from the main document to the supplement.

\begin{appendix}

\section{Additional material and technical results}

\subsection{Properties of Gaussian priors and proof of Lemma \ref{Lem:Conjugacy}}

\subsubsection{Periodic Matérn processes}
\label{Sec:PerMatProc}

	The Matérn process on $\R^d$ with smoothness parameter $s+1-d/2>0$ is a stationary Gaussian process with covariance kernel (Example 11.8 in \cite{ghosal2017})
\begin{equation}
\label{Eq:MaternKernel}
	K(x,y) \equiv K(x-y) = \int_{\R^d} e^{-i (x-y).\xi}(1+\|\xi\|^2)^{-s-1}
	d\xi,\qquad x,y\in \R^d.
\end{equation}
Using a standard approach to periodization \cite{RW06}, one can construct a corresponding stationary periodic kernel via the Poisson summation formula:
$$
	K_{\textnormal{per}}(x,y)
	\equiv K_{\textnormal{per}}(x-y) =  \sum_{m\in\Z^d} K(x-y+m), 
	\qquad x,y\in\R^d.
$$
Since the Poisson summation formula preserves the Fourier transform (Theorem 8.31 of \cite{F99}), $K_{\textnormal{per}}$ has Fourier coefficients
$$
	\int_{\T^d} e^{-2\pi i k.u}K_{\textnormal{per}}(u)du
	= \int_{\R^d} e^{-2\pi i k.u}K(u)du
	= \frac{(2\pi)^d}{(1+4\pi^2 \|k\|^2)^{s+1}}, \qquad k\in\Z^d,
$$
the last equality following from \eqref{Eq:MaternKernel} and the Fourier inversion formula (e.g., Section 8.3 in \cite{F99}). Thus, using the Fourier series of $K_{\textnormal{per}}(u)$, the periodized Matérn kernel has series representation 
$$
	K_{\textnormal{per}}(x,y) 
	= (2\pi)^d \sum_{k\in\Z^d} \frac{1}{(1+4\pi^2 \| k\|^2)^{s+1}} e_k(x)
	\overline{e_k(y)},
$$
where $\{e_k, \ k\in\Z^d\}$ is the Fourier basis of $L^2(\T^d)$, i.e.~$e_k(x) = e^{2\pi i  k . x}$. Theorem I.21 in \cite{ghosal2017} then implies that the centred Gaussian process $W=\{W(x): x\in\T^d\}$ with covariance kernel $K_{\textnormal{per}}$ has RKHS
$$
	\H = \left\{ h = \sum_{k\in\Z^d} h_k e_k :\ \|h\|^2_\H =
	(2\pi)^d\sum_{k\in\Z^d} h_k^2 (1+4\pi^2 \|k\|^2)^{s+1}<\infty\right\}.
$$
By the Fourier series characterisation of Sobolev spaces, we thus have $\H = H^{s+1}(\T^d)$ and $\|\cdot\|_\H$ is an equivalent norm to $\|\cdot\|_{H^{s+1}}$. The above computation also implies that the periodic Mat\'ern process has Karhunen-Loève series expansion given by \eqref{Eq:MaternKLSeries}, equivalent (up to constants) to the mean-zero Gaussian processes with inverse Laplacian covariance appearing in \cite{PPRS12,PSvZ13,vWvZ16,vW2019}.

	For the regularity of the sample paths, by Proposition I.4 in \cite{ghosal2017}, which applies also to the periodized kernel $K_{\textnormal{per}}$, the periodic Matérn process has a version $W$ whose sample paths are in $C^{s+1-\frac{d}{2}-\eta}(\T^d)$ for any $\eta>0$. By Lemma I.7 in \cite{ghosal2017}, this version then defines a Borel random element in $C^{r}( \T^d)$ for all $r<s+1-d/2$. In particular, $C^r(\T^d)$ is a separable linear subspace of $C^{(d/2+\kappa)\vee 2}( \T^d)$ for some $\kappa>0$ provided that $s+1>d/2+(d/2)\vee 2$ and $r>(d/2+\kappa)\vee 2$. In conclusion, the periodic Mat\'ern process satisfies Condition \ref{GP_condition} for $s+1>d/2+(d/2)\vee 2$.

\subsubsection{Gaussian conjugacy formulae and proof of Lemma \ref{Lem:Conjugacy}}
\label{Sec:Conjugacy}

We first provide details for the derivation of the conjugate formula \eqref{Eq:TruncatedPosterior}. For basis functions $(h_k, \ k\in\N)\subset C^2\cap \dot L^2$, and fixed $K\in\N$, identifying a function $B = \sum_{k=1}^K B_k h_k$ with its coefficient vector $\mathbf B = (B_1,\dots,B_K)^T\in\R^K$, we can write the log-likelihood \eqref{Eq:LogLik} in quadratic form as
\begin{align*}
	\ell_T(\mathbf B ) 
	&=
		-\frac{1}{2} \int_0^T \sum_{j=1}^d\sum_{k,k'} B_k B_{k'} \partial_{x_j} h_k
		(X_t) \partial_{x_j} h_{k'}(X_t) dt
		+ \sum_{j=1}^d \int_0^T \sum_{k=1}^K B_k \partial_{x_j}h_k
		(X_t) dX_t^j\\
	&= 
		-\frac{1}{2}\sum_{k,k'}B_k B_{k'} \left[\int_0^T \nabla h_k
		(X_t). \nabla h_{k'}(X_t) dt	\right]
		+ \sum_{k=1}^K B_k \left[\int_0^T \nabla h_k
		(X_t). dX_t\right]\\
	&=
		-\frac{1}{2} \mathbf B^T \Sigma \mathbf B
		+  \mathbf B^T \mathbf H,
\end{align*}
for $\Sigma$ and $\mathbf H$ defined as after \eqref{Eq:TruncatedPosterior}. Under the same identification, we may regard the prior \eqref{Eq:TruncatedKL} as a multivariate normal distribution on $\R^K$ with diagonal covariance matrix $\Upsilon = \textnormal{diag}(\upsilon_1^2,\dots,\upsilon_K^2)$. Then, by Bayes' formula, the posterior density is given by
\begin{align*}
	d\Pi(\mathbf B|X^T) 
	&\propto 
		e^{\ell_T(\mathbf B)}
		e^{- \frac{1}{2}\mathbf B^T \Upsilon^{-1} \mathbf B}
	=
		e^{-\frac{1}{2} \mathbf B^T \Sigma \mathbf B
		+  \mathbf B^T \mathbf H - \frac{1}{2} \mathbf B^T \Upsilon^{-1}
		\mathbf B}.
\end{align*}
Completing the squares then gives
\begin{align*}
	d\Pi(\mathbf B|X^T) 
	& \propto
		e^{-\frac{1}{2} [\mathbf B-(\Sigma+\Upsilon^{-1})^{-1}
		\mathbf H]^T(\Sigma+\Upsilon^{-1}) [\mathbf B 
		-(\Sigma+\Upsilon^{-1})^{-1}\mathbf H]},
\end{align*}
which completes the derivation of \eqref{Eq:TruncatedPosterior}. We next provide the proof of Lemma \ref{Lem:Conjugacy}.

\begin{lemmaConj}
	Let $\Pi$ be a centred Gaussian Borel probability measure on $L^2(\T^d)$ that is supported on $C^2(\T^d)\cap \dot L^2(\T^d)$. Then the posterior distribution \eqref{Eq:Posterior} is almost surely (under the law of the data $X^T$) Gaussian on  $L^2(\T^d)$.
\end{lemmaConj}

\begin{proof}
We follow ideas in \cite{PSvZ13}. For $K\in\N$, let $P_K: L^2\to L^2$ be the projection onto the Fourier approximation space $E_K = \textnormal{span}\{e_k, \ |k|\le K\}$, $e_k(x) = e^{2\pi i k.x}$. Consider the approximate posterior
$$
	\Pi_K(A|X^T) = 
	\frac{\int_A e^{\ell_T(P_KB)}d\Pi(B)}{\int_{C^2} e^{\ell_T(P_KB')}d\Pi(B')},
	\qquad A\subset C^2\cap \dot L^2 \ \textnormal{measurable}.
$$
Since $\ell_T(P_K B)$ only depends on $P_KB$, $\Pi_K(\cdot|X^T)$ can be decomposed as the product of a Gaussian measure on $E_K$ (satisfying the conjugate formula \eqref{Eq:TruncatedPosterior}) and a Gaussian measure obtained as the push-forward of $\Pi$ under the projection operator onto the orthogonal complement of $E_K$ in $L^2$. It follows that $\Pi_K(\cdot|X^T)$ is Gaussian in $L^2$. We now show that, a.s.~under the law of the data $X^T$, $\Pi_K(\cdot|X^T) \to \Pi(\cdot|X^T)$ weakly in $L^2$ as $K\to\infty$, which in turn will conclude the proof since the class of Gaussian measures is closed with respect to weak convergence.

	For $B\in C^2$, we have $\|B -P_K B\|_{H^2}\to 0$ as $K\to\infty$. It follows that $e^{\ell_T(P_KB)}\to e^{\ell_T(B)}$ a.s., since, given the data $X^T$, the function $B\mapsto \ell_T(B)$ in \eqref{Eq:LogLik} is continuous with respect to the $H^2$-norm. We next show that the limiting function $e^{\ell_T(B)}, \ B\in C^2,$ is dominated by a $\Pi$-integrable function. Under $P_{B_0}$, using It\^o's lemma (Theorem 39.3 of \cite{B11}),
\begin{align*}
	\ell_T(B) 
	&\le 
		\int_0^T \nabla B(X_t).dX_t\\
	&=
	B(X_T)- B(X_0) - \frac{1}{2}\int_0^T \Delta B(X_t)dt  \\
	& \le 
		2\|B\|_\infty + \frac{1}{2}\sum_{j=1}^d
		\int_0^T \Big| \frac{\partial^2}
		{\partial x_j^2}B(X_t)\Big|dt,
\end{align*}
and using Young's inequality with $\varepsilon$ and $p=q=2$, we obtain that for all $\eta>0$,
\begin{align*}
	\ell_T(B)
	&\le
		\frac{2}{\eta} + \frac{\eta}{2}\|B\|_\infty^2 + 
		\frac{1}{2}\sum_{j=1}^d\int_0^T \left[
		\frac{1}{2\eta}
		+\frac{\eta}{2}\Big|\frac{\partial^2}
		{\partial x_j^2}B(X_t)\Big|^2
		 \right]dt	\\
	&\le 
		\frac{2}{\eta} + \frac{\eta \|B\|_\infty^2}{2} +
		\frac{dT}{4\eta} + \frac{dT\eta}{4}\|B\|_{C^2}^2
	=
		\frac{dT+8}{4\eta}
		+\frac{\eta(dT+2)}{4}\| B\|_{C^2}^2.
\end{align*}
In conclusion, for all $\eta>0$,
$$
	e^{\ell_T(B)}\le e^{(dT+8)/(4\eta)} 
	e^{\eta(dT+2)\| B\|_{C^2}^2/4}.
$$
Taking $\eta>0$ small enough, since $\Pi$ can be regarded as a Gaussian measure on $C^2$, Fernique's theorem \cite[Theorem 2.8.5]{B98} implies that the right hand side in the last display is $\Pi$-integrable. By the dominated convergence theorem, we then conclude that a.s.~under the law of the data $X^T$, as $K\to\infty$,
$$
	\int_{C^2} e^{\ell_T(P_KB')}d\Pi(B') 
	\to \int_{C^2}  e^{\ell_T(B')}d\Pi(B'),
$$
and likewise, for all measurable $A\subset C^2\cap \dot L^2$, by boundedness of the indicator function $1_A(\cdot)$,
$$
	\int_A e^{\ell_T(P_KB)}d\Pi(B)
	=
		\int_{C^2}  e^{\ell_T(P_KB)}1_A(B) d\Pi(B)
	\to
	 \int_A e^{\ell_T(B)}d\Pi(B),
$$	
showing as required that $\Pi_K(\cdot|X^T) \to \Pi(\cdot|X^T)$ weakly in $C^2$ (and also in $L^2$). 
\end{proof}

\subsection{Proof of Theorem \ref{thm_contraction_tests}}
\label{Sec:ProofTestTheo}

\begin{theoConv}
	Let $d_T$ be a semimetric on the parameter space $\mathcal{H} \subseteq C^1(\T^d)$ for the drift $b$ and let $\Pi = \Pi_T$ be priors for $b$. Let $\eps_T\to 0$ satisfy $\sqrt{T}\eps_T \to \infty$, let $\xi_T \to 0$, $\mathcal{H}_T \subseteq \mathcal{H}$ and let $\phi_T$ be a sequence of test functions satisfying
\begin{align*}
	P_{b_0} \phi_T \to 0, \qquad \qquad 
	\sup_{b \in \mathcal{H}_T :  d_T(b,b_0) \geq D \xi_T} P_b (1-\phi_T) \leq Le^{-(C+4)T\eps_T^2}
\end{align*}
for some $C,D,L>0$ with $\Pi(\mathcal{H}_T^c) \leq e^{-(C+4)T\eps_T^2}$. Suppose further that there exist deterministic sets $\mathcal{SB}_T \subseteq \mathcal{H}$ with $\Pi(\mathcal{SB}_T) \geq e^{-CT\eps_T^2}$ and
$$P_{b_0} \left(\sup_{b\in \mathcal{SB}_T}  \int_0^T \|b(X_s)-b_0(X_s)\|^2 ds \leq T \eps_T^2 \right) \to 1,$$
where $b_0$ is the true drift function. Then for $M>0$ large enough, as $T\to\infty$,
$$
	P_{b_0} \Pi (b:d_T(b,b_0) \geq M\xi_T|X^T) \to 0.
$$
\end{theoConv}

\begin{proof}
The proof follows by the standard arguments for test-based contraction rates, see e.g.~Theorem 2.1 of \cite{ghosal2000}, with the only difference being how we control the denominator in the Bayes formula. We now detail this argument, which is a modification of Lemma 4.2 of van der Meulen et al. \cite{vdMeulen2006} adapted to our setting. By \eqref{Eq:LogLik}, the log-likelihood ratio equals (replacing $\nabla B$ with $b$), under $P_{b_0}$,
$$
	\log \frac{dP_b}{dP_{b_0}}(X^T) 
	= \int_0^T [b(X_s) - b_0(X_s)].dW_s - \tfrac{1}{2}\int_0^T \|b(X_s)
	-b_0(X_s)\|^2 ds =: M_T^b - \tfrac{1}{2} [M^b]_T.
$$
Set $\overline{\Pi} = \Pi(\cdot \cap \mathcal{SB}_T) /\Pi(\mathcal{SB}_T)$ to be the normalized prior restricted to the set $\mathcal{SB}_T$. Since $P_{b_0} (\sup_{b \in \mathcal{SB}_T} [M^b]_T \leq T\eps_T^2) \to 1$ by assumption, we henceforth work on this event. By Jensen's inequality,
\begin{align*}
	\log \int \frac{dP_b}{dP_{b_0}}(X_T) \frac{d\Pi(b)}{\Pi(\mathcal{SB}_T)} 
	& \geq \int_{\mathcal{SB}_T} \log \frac{dP_b}{dP_{b_0}}(X_T) d\overline{\Pi}(b) 
	= \int_{\mathcal{SB}_T} M_T^b - \tfrac{1}{2}[M^b]_T d\overline{\Pi}(b).
\end{align*}
For fixed $T>0$, define
$$Z_t^T = \int_{\mathcal{SB}_T} M_t^b d\overline{\Pi}(b) = \sum_{i=1}^d \int_0^t  \int_{\mathcal{SB}_T} b_i(X_s)-b_{0,i}(X_s) d\overline{\Pi}(b) dW_s^i,$$
where the last equality follows from stochastic Fubini's theorem (Theorem 64 of Chapter IV in \cite{protter2004}) since $\mathcal{SB}_T$ is deterministic. The process $(Z_t^T:t \geq 0)$ is a continuous local martingale under $P_{b_0}$ with quadratic variation
\begin{align*}
[Z^T]_t & = \int_0^t \sum_{i=1}^d \left( \int_{\mathcal{SB}_T} b_i(X_s)-b_{0,i}(X_s) d\overline{\Pi}(b) \right)^2 ds.
\end{align*}
Using Jensen's inequality, Fubini's theorem and the definition of $\mathcal{SB}_T$,
\begin{align*}
[Z^T]_T & \leq \int_0^T \sum_{i=1}^d \int_{\mathcal{SB}_T} (b_i(X_s)-b_{0,i}(X_s) )^2 d\overline{\Pi}(b) \, ds \\
& = \int_{\mathcal{SB}_T} \int_0^T  \|b(X_s)-b_{0}(X_s) \|^2 ds \, d\overline{\Pi}(b) = \int_{\mathcal{SB}_T} [M^b]_T d\overline{\Pi}(b)
\leq T\eps_T^2.
\end{align*}
By Bernstein's inequality \eqref{Bernstein}, for any $x>0$,
\begin{align*}
P_{b_0} \left( |Z_T^T| \geq x \right) & = P_{b_0} \left( |Z_T^T| \geq x,[Z^T]_T \leq T\eps_T^2 \right)  \leq 2\exp \left(-\frac{x^2}{2T\eps_T^2} \right).
\end{align*}
Setting $x=LT\eps_T^2$ gives $P_{b_0}(|Z_T^T|\geq LT\eps_T^2) \to 0$ for any $L>0$. On the event $\{|Z_T^T|\leq LT\eps_T^2\}$, which has $P_{b_0}$-probability tending to one, and using the second to last display,
\begin{align*}
\int_{\mathcal{SB}_T} M_T^b - \tfrac{1}{2}[M^b]_T d\overline{\Pi}(b) & = Z_T^T-  \tfrac{1}{2} \int_{\mathcal{SB}_T} [M^b]_T d\overline{\Pi}(b) \geq -(L+1/2) T\eps_T^2.
\end{align*}
In conclusion, we have shown $P_{b_0}( \int dP_b/dP_{b_0} d\Pi(b) \geq e^{-(L+1/2)T\eps_T^2} \Pi(\mathcal{SB}_T)) \to 1$ for any $L>0$. Substituting in $\Pi(\mathcal{SB}_T) \geq e^{-CT\eps_T^2}$ gives $P_{b_0}( \int dP_b/dP_{b_0} d\Pi(b) \geq e^{-(C+L+1/2)T\eps_T^2}) \to 1$ for any $L>0$.
\end{proof}

\subsection{Proofs of convergence rates for posterior means}

\begin{CorollaryPostMean1}
Let $\hat{B}_T =E^{\Pi_T}[B|X^T]$ be the posterior mean/MAP estimate. Under the conditions of Theorem \ref{Prop:GaussRates}, as $T \to \infty$,
$$\|\nabla \hat{B}_T - \nabla B_0\|_2 = O_{P_{B_0}}(T^{-s/(2s+d)} ).$$ 
\end{CorollaryPostMean1}

\begin{proof}
By Lemma \ref{Lem:Conjugacy}, the posterior $\Pi(\cdot|X^T)$ of $B$ is Gaussian, $P_{B_0}$-almost surely. Since the set $\{B \in \dot{L}^2: \|\nabla B\|_2 \leq MT^{-s/(2s+d)} \}$ is closed, convex and symmetric, Theorem \ref{Prop:GaussRates} and then Anderson's inequality imply that
\begin{align*}
3/4 & < \Pi(B:\|\nabla B-\nabla B_0\|_2 \leq MT^{-s/(2s+d)} |X^T) \\
& \leq \Pi(B:\|\nabla B-\nabla \hat{B}_T\|_2 \leq MT^{-s/(2s+d)} |X^T)
\end{align*}
for $M,T>0$ large enough with $P_{B_0}$-probability tending to one. In particular, the posterior probability of the intersection of the above two sets is at least $1/2$, and so there exists $B \in \dot{L}^2$ in both sets. For such a $B$,
$$\|\nabla \hat{B}_T - \nabla B_0\|_2 \leq \|\nabla \hat{B}_T - \nabla B\|_2 + \|\nabla B - \nabla B_0\|_2 \leq 2MT^{-s/(2s+d)},$$
which thus holds with $P_{B_0}$-probability tending to one.
\end{proof}

	The proof for (the non-conjugate) $p$-exponential priors follows from uniform integrability arguments developed in \cite{MNP20}.

\begin{CorollaryPostMean2}
Let $\hat{B}_T =E^{\Pi_T}[B|X^T]$ be the posterior mean. Under the conditions of Theorem \ref{Prop:PExpRates}, as $T \to \infty$,
$$\|\nabla \hat{B}_T - \nabla B_0\|_p = O_{P_{B_0}}(T^{-s/(2s+d)} ).$$ 
\end{CorollaryPostMean2}

\begin{proof}

Inspection of the proof of Theorem \ref{thm_contraction_tests} shows that the assertion of Theorem \ref{Prop:PExpRates} can be strengthened as follows: under the same assumptions, for any $A>0$ we can find large enough $M>0$ such that for $\eps_T = T^{-s/(2s+d)}$,
\begin{equation}
\label{Eq:StrongThesis}
	P_{B_0}\left(\Pi( B : \|\nabla B - \nabla B_0\|_{p} > M\eps_T| X^T) 
	> e^{-AT\eps_T^2}\right) \to 0
\end{equation}
as $T\to\infty$. 
Using Jensen's inequality,
\begin{align*}
	\| \nabla \hat B_T - \nabla B_0\|_{p}
	&\lesssim
		\| E^\Pi[B|X^T] -  B_0\|_{W^{1,p}}
	\le
		E^\Pi\left[ \| B - B_0\|_{W^{1,p}}|X^T\right].
\end{align*}
For any $M>0$, the last quantity is bounded by
\begin{align*}
	&
		M\eps_T +
		E^\Pi\left[ \| B - B_0\|_{W^{1,p}}1_{\{\| B - B_0\|_{W^{1,p}}>M\eps_T\}}
		\big|X^T\right] \\
	&\quad \le		
			M\eps_T + 	E^\Pi\left[ \| B - B_0\|_{W^{1,p}}^p|X^T\right]^{1/p}
		\Pi(B : \| B - B_0\|_{W^{1,p}}>M\eps_T|X^T)^{1/q},
\end{align*}
where we have used H\"older's inequality with $1/p+1/q=1$.
Under the identifiability assumption $B,B' \in \dot L^2(\T^d)$, Poincaré's inequality implies that the norm $\|\nabla B - \nabla B'\|_p$ is equivalent to the Sobolev norm $\|B - B'\|_{W^{1,p}}$. Using the above and \eqref{Eq:StrongThesis}, for any $A>0$, we can find $M>0$ large enough such that
\begin{align*}
	& P_{B_0} (\| \nabla \hat B_T - \nabla B_0\|_p> 2M\eps_T) \\
	& \leq P_{B_0} \left( E^\Pi\left[ \| B - B_0\|_{W^{1,p}}^p|X^T\right]^{1/p}
		\Pi(B : \| B - B_0\|_{W^{1,p}}>M\eps_T|X^T)^{1/q} > M\eps_T\right)\\
	& \leq P_{B_0} \left( E^\Pi\left[ \| B - B_0\|_{W^{1,p}}^p|X^T\right]
		e^{-pAT\eps_T^2/q} > M^p \eps_T^p \right) + o(1).
\end{align*}
For $C>0$ the constant in the small ball estimate \eqref{Eq:NthSmallBall}, define the event
$$
	A_T = \left\{ \int_{C^2\cap \dot L^2}
	e^{\ell_T(B)-\ell_T(B_0)}d\Pi(B) \geq e^{-(C+1)T\eps_T^2} \right\},
$$
which is shown to satisfy $P_{B_0}(A_T)\to 1$ as $T\to\infty$ in the proof of Theorem \ref{thm_contraction_tests}. The second last display is then upper bounded by
\begin{align*}
	P_{B_0}&\left( \left\{ \frac{\int_{C^2\cap \dot L^2}\| B - B_0\|^p_{W^{1,p}}
	e^{\ell_T(B)-\ell_T(B_0)}d\Pi(B)} 
	{\int_{C^2\cap \dot L^2}
	e^{\ell_T(B)-\ell_T(B_0)}d\Pi(B)}
	>M^p \eps_T^p e^{pAT\eps_T^2/q} \right\} \cap  A_T\right) + o(1)\\
	&\le
		P_{B_0}\left(\int_{C^2\cap \dot L^2}\| B - B_0\|^p_{W^{1,p}}
		e^{\ell_T(B)-\ell_T(B_0)}d\Pi(B)
		>M^p \eps_T^p e^{(pA/q-C-1)T\eps_T^2}\right) + o(1).
\end{align*}
By Markov's inequality and Fubini's theorem, the latter probability is smaller than
\begin{align*}
		& \frac{e^{-(pA/q-C-1)T\eps_T^2}}{M^p \eps_T^p} 
		 \int_{C^2\cap \dot L^2}\| B - B_0\|^p_{W^{1,p}}
		E_{B_0}  \left[ e^{\ell_T(B)-\ell_T(B_0)} \right] d\Pi(B) \\
		&\qquad =
		  \frac{e^{-(pA/q-C-1)T\eps_T^2}}{M^p \eps_T^p} E^\Pi \| B - B_0\|^p_{W^{1,p}},
\end{align*}
since $E_{B_0}\left[e^{\ell_T(B)-\ell_T(B_0)}\right]=E_{B}[1]=1$. Using the definition \eqref{Eq:pExpPrior} of the $p$-exponential prior $\Pi$, the continuous embedding $B_{p1}^1 \subset W^{1,p}$ and that $|\sum_{i=1}^m a_i|^p \leq m^{p-1} \sum_{i=1}^m |a_i|^p$,
\begin{align*}
	E^\Pi \|B\|_{W^{1,p}}^p
	\lesssim
	E^\Pi \|B\|_{B^1_{p1}}^p 
	& = T^{-\frac{d}{2s+d}}  E \left[ \sum_{l=0}^J 
	2^{l(1+\frac{d}{2}-\frac{d}{p})}
	\left( \sum_{r} 2^{-pl(s+1+\frac{d}{2}-\frac{d}{p})}|\rho_{lr}|^p \right)^\frac{1}{p} \right]^p \\
	& \leq T^{-\frac{d}{2s+d}} (J+1)^{p-1}  \sum_{l=0}^J 
	2^{pl(1+\frac{d}{2}-\frac{d}{p})} \sum_{r} 2^{-pl(s+1+\frac{d}{2}-\frac{d}{p})} E|\rho_{lr}|^p  \\
	& \lesssim T^{-\frac{d}{2s+d}} J^{p-1}  \sum_{l=0}^J 2^{-l(ps-d)} \lesssim T^{-\frac{d}{2s+d}} J^{p-1}  \lesssim 1,
\end{align*}
since $2^J \simeq T^{1/(2s+d)}$ and $s > d/p + (d/2)\vee2 > d/p$ by assumption. Combining the last three displays then yields that for any $A>0$, we can find $M>0$ large enough such that  
\begin{align*}
P_{B_0} (\| \nabla \hat B_T - \nabla B_0\|_p> 2M\eps_T) \lesssim \frac{e^{-(pA/q-C-1)T\eps_T^2}}{M^p \eps_T^p} + o(1)
\end{align*}
since $B_0 \in W^{1,p}$. Taking $A>(C+1)q/p$ and using that $\eps_T = T^{-s/(2s+d)}$, the right-hand side tends to zero as $T \to\infty$.
\end{proof}


\subsection{A PDE estimate}

In the proofs we used a regularity estimate for solutions to the Poisson equation $L_b u = f$, where the generator $L_b$ of the diffusion is the strongly elliptic second order partial differential operator given in \eqref{Eq:generator}. For some basic facts about this elliptic PDE, see Section 6 in \cite{nicklray2020}, while for more general theory for periodic elliptic PDEs, see Chapter II.3 in \cite{BJS64}. The following estimate is only proved for $t\geq 2$ in Lemma 11 of \cite{nicklray2020}, but the proof can be extended to general $t\in \R$.

\begin{lemma}\label{regest}
Let $t \in \R$ and assume $b \in C^{|t-2|}(\T^d)$. For any $f \in \dot{L}^2_{\mu_b}(\T^d)$, there exists a unique solution $L_b^{-1}[f] \in  L^2(\T^d)$ of the equation
$$L_b u=f,\qquad f\in L^2(\T^d),$$
satisfying $L_b L_b^{-1}[f]=f$ almost everywhere and $\int_{\T^d} L_b^{-1}[f](x) dx = 0$. Moreover, $$\|L_b^{-1}[f]\|_{H^t} \lesssim \|f\|_{H^{t-2}},$$
with constants depending on $t,d$ and on an upper bound for $\|b\|_{B^{|t-2|}_{\infty \infty}}$.

\proof 

Recall the multiplication inequality for Besov-Sobolev norms with $s\geq 0$ (\cite{T10}, p. 143),
\begin{equation*}
\|fg\|_{B_{pq}^s} \leq c(s,p,q,d) \|f\|_{B_{pq}^s} \|g\|_{B_{\infty\infty}^s} \leq c'(s,p,q,d) \|f\|_{B_{pq}^s} \|g\|_{C^s},
\end{equation*}
which by duality implies that for $s\geq 0$,
\begin{equation*}
\|fg\|_{B_{pq}^{-s}} = \sup_{\|\varphi\|_{B_{p'q'}^s} \leq 1} \left| \int_{\T^d} f\frac{g\varphi}{\|g\varphi\|_{B_{p'q'}^s}} \right| \|g\varphi\|_{B_{p'q'}^s} \leq c(s,p,q,d) \|f\|_{B_{pq}^{-s}} \|g\|_{B_{\infty\infty}^s},
\end{equation*}
where $1/p+1/p'=1/q + 1/q' =1$. Combining the last displays yields the multiplier inequality
\begin{equation}\label{besov_mult}
\|fg\|_{B_{pq}^s} \leq c(|s|,p,q,d) \|f\|_{B_{pq}^s} \|g\|_{B_{\infty\infty}^{|s|}}, ~~~ s\in \R.
\end{equation}
The proof of the result then follows as in Lemma 11 of \cite{nicklray2020} upon using the multiplier inequality \eqref{besov_mult} instead of (3) when establishing (69) in that paper. We omit the details here.
\endproof 

\end{lemma}

\subsection{The Gaussian correlation inequality}

	The Gaussian correlation inequality states that for any closed, symmetric, convex sets $K, L$ in $\R^d$ and any centred Gaussian measure $\mu$ in $\R^d$, we have
$$
\mu(K\cap L)\ge \mu(K)\mu(L).
$$
This was proved by Royen \cite{R14}; see also \cite{LM17} for a self-contained presentation and proof. In the proof of Lemma \ref{lem:prior_prob_event} we have used the following extension to Gaussian measures on separable Banach spaces, which was referred to, without full details, in Remark 3 (i) in \cite{LM17}. We include a proof for completeness.

\begin{lemma}\label{lem:InfiniteIneq}
Let $\mu$ be a centred Borel Gaussian measure on a separable Banach space $\B$. Let $K,L$ be closed (with respect to the topology induced by $\|\cdot\|_\B$), convex, symmetric subsets of $\B$. Then,
$$
	\mu(K\cap L)\ge \mu(K)\mu(L).
$$

\proof
For $X\sim \mu$, the Karhunen-Loève expansion (\cite{ginenickl2016}, Theorem 2.6.10) of $X$ implies that there exists a complete orthonormal system $\{x_n, \ n\ge1\}$ of the RKHS of $X$ and i.i.d.~standard normal random variables $(\xi_n)_{n\ge1}$, such that $X_n=\sum_{i=1}^n\xi_ix_i$ converges almost surely to $X$ in the norm of $\B$. We first show that for any closed, convex, symmetric set $K \subseteq \B$,
\begin{equation}\label{eq:KL_expansion}
\mu(K) = \Pr \left( \sum_{i=1}^\infty \xi_i x_i \in K \right) = \lim_{n\to\infty} \Pr \left( \sum_{i=1}^n \xi_i x_i \in K \right).
\end{equation}
Since $K$ is closed and $X_n \to X$ almost surely, and hence also in law, the Portmanteau lemma (\cite{V98}, Lemma 2.2) implies
$$
	\limsup_{n\to\infty}\Pr(X_n\in K)\le \Pr(X\in K).
$$ 
On the other hand, since $K$ is convex and symmetric, Anderson's inequality (\cite{ginenickl2016}, p. 49) yields
$\Pr( X\in K) \leq \Pr(X_n\in K)$ for all $n\ge1$. Together, these establish \eqref{eq:KL_expansion}.

	Now let $K,L \subseteq \B$ be arbitrary closed, convex, symmetric sets and define
$$K_n=\left\{z=(z_1,\dots,z_n)\in\R^n : \sum_{i=1}^nz_ix_i\in K\right\} \subseteq \R^n$$
and analogously $L_n$. It is straightforward to check that $K_n$ is a convex and symmetric subset of $\R^n$. Furthermore, if $\{z^{(k)}\}_{k\geq 1} \subset K_n$ converges to some $z = (z_1,\dots,z_n) \in\R^n$, then
$$
	\left\|\sum_{i=1}^n z^{(k)}_i x_i - \sum_{i=1}^n z_i x_i\right\|_\B
	\le 
		\max_{i=1,\dots,n}\|x_i\|_\B \sum_{i=1}^n |z^{(k)}_i - z_i|
	\to
		0,
	\quad
		k\to\infty,
$$
so that, $K$ being closed in $\B$, $\sum_{i=1}^n z_i x_i\in K$ and $z\in K_n$. This shows that $K_n$ is closed, convex and symmetric in $\R^n$. Denoting by $\gamma_n=\Lcal(\xi_1,\dots,\xi_n)$ the standard Gaussian measure on $\R^n$, \eqref{eq:KL_expansion} then implies $\mu(K) = \lim_{n\to\infty} \gamma_n(K_n)$. Note that identical considerations apply to $L$ and $L_n$ and, since $K \cap L$ is also closed, convex and symmetric in $\B$, also to $K \cap L$ and $K_n \cap L_n$.

	Applying this and the Gaussian correlation inequality for the finite-dimensional Gaussian measure $\gamma_n$,
$$
	\mu(K\cap L)= \lim_{n\to\infty}\gamma_n(K_n\cap L_n)
	\ge \lim_{n\to\infty}\gamma_n(K_n)\gamma_n(L_n)
	= \mu(K)\mu(L).
$$
\endproof

\end{lemma}

\section{Proofs of Lemmas  \ref{Lemma:AnotherSmallBall}, \ref{lem:prior_prob_event} and \ref{lem:bias}}
\label{ProofOfProof}

\subsection{Proof of Lemma \ref{Lemma:AnotherSmallBall}}
\label{Sec:ProofOtherSmallBall}

\begin{lemmaSmallBall}
Suppose $B_0 \in H^{s+1}(\T^d) \cap \dot{L}^2(\T^d) $ for $s > d \vee (3/2)$. If $\eps_T = T^{-s/(2s+d)}$, $J \in \mathbb{N}$ satisfies $2^J \simeq T^{1/(2s+d)}$ and $M>0$, then as $T\to\infty$,
\begin{align*}
	P_{B_0} & \left(  \sup_{B \in V_J \cap \dot{L}^2: \|\nabla B - \nabla B_0\|
	_{L^2(\mu_0)} \leq M\eps_T} \frac{1}{T} \int_0^T  
	\|\nabla B(X_s)-\nabla B_0(X_s)\|^2 ds \leq M^2 \eps_T^2 
	+ o(\eps_T^2) \right) \\
	& \qquad \qquad \to 1.
\end{align*}
\end{lemmaSmallBall}

\begin{proof}

We first show that $\tfrac{1}{T} \int_0^T  \|\nabla B(X_s)-\nabla B_0(X_s)\|^2 ds$ is close to its ergodic average $\|\nabla B - \nabla B_0\|_{L^2(\mu_0)}^2$, uniformly over $ \Scal\Bcal_T  = \{B \in V_J\cap \dot{L}^2: \|\nabla B - \nabla B_0\|_{L^2(\mu_0)} \leq M \eps_T\}$. Define
$$
	\mF_T = \{ f_B(x) := \|\nabla B(x) - \nabla B_0(x)\|^2 
	- \|\nabla B - \nabla B_0\|_{L^2(\mu_0)}^2 : B \in  \Scal\Bcal_T  \cup \{B_0\} \}.
$$
Since $x \mapsto \|x\|^2$ is a smooth function, $ \Scal\Bcal_T  \subset V_J$ and $B_0 \in H^{s+1}$, it holds that $\mF_T \subset \dot{L}_{\mu_0}^2 \cap H^s$. By Lemma \ref{regest}, the Poisson equation $L_{B_0} u = f_B$ (writing $L_{B_0}$ for $L_{\nabla B_0}$) has a unique solution $L_{B_0}^{-1}[f_B]\in \dot L^2$ for any $B \in  \Scal\Bcal_T $. Applying Lemma 1 of \cite{nicklray2020}, since $s> d\vee(3/2) >(d/2+\kappa)\vee1 $ for $\kappa>0$ small enough, gives that for any $x\geq 0$,
\begin{equation}
\label{chaining_inequality}
	P_{B_0} \left( \sup_{f_B \in \mF_T} | \G_T[f_B]| 
	\geq \sup_{f\in\mF_T} \frac{2\|L_{B_0}^{-1}[f_B]\|_\infty}{\sqrt{T}} 
	+ J_{\mF_T}(4\sqrt{2} + 192x) \right) \leq e^{-x^2/2},
\end{equation}
where $J_{\mF_T} = \int_0^{D_{\mF_T}} \sqrt{ \log 2N(\mF_T,6d_L,\tau)} d\tau$, $d_L^2(f,g) = \sum_{i=1}^d \|\partial_{x_i} L_{B_0}^{-1}[f-g]\|_\infty^2$ and $D_{\mF_T}$ is the $d_L$-diameter of $\mF_T$. We now proceed to bound $J_{\mF_T}$.

	For $B \in  \Scal\Bcal_T $, write $h_i = \partial_{x_i}(B-B_0)$ so that $f_B = \sum_{i=1}^d h_i^2 - \|h_i\|_{L^2(\mu_0)}^2$. Using the Sobolev embedding theorem, Lemma \ref{regest} and the Runst-Sickel lemma (\cite{nicklray2020}, Lemma 2), for any $\kappa>0$,
\begin{align*}
	d_L(f_B,f_{\bar{B}} ) 
	& \lesssim \|f_B - f_{\bar{B}}\|_{H^{d/2+\kappa-1}} \\
	&\lesssim \sum_{i=1}^d \Big{\|} h_i^2 - \bar{h}_i^2 - \int_{\T^d} 
	(h_i^2 - \bar{h}_i^2 ) d\mu_0 \Big{\|}_{H^{d/2+\kappa-1}}\\
	&\lesssim \sum_{i=1}^d \| h_i -\bar{h}_i\|_ \infty \|h_i 
	+ \bar{h}_i\|_{H^{(d/2+\kappa-1)_+}} 
	+ \| h_i -\bar{h}_i\|_{H^{(d/2+\kappa-1)_+}} \|h_i + \bar{h}_i\|_\infty \\
	& \qquad  + \|\mu_0\|_\infty \|h_i-\bar{h}_i\|_2 \|h_i 
	+ \bar{h}_i\|_2 \|1\|_{H^{d/2+\kappa-1)}},
\end{align*}
where the constants depend only on $d$, $\kappa$ and $\|\nabla B_0\|_{B_{\infty\infty}^{|d/2+\kappa-1|}} \lesssim \|B_0\|_{H^{|d/2+\kappa-1|+d/2+1}} \lesssim \|B_0\|_{H^{s+1}}$, and we note that $\|\mu_0\|_\infty \leq e^{4\|B_0\|_\infty} <\infty$ using \eqref{Eq:InvMeas}. Since $B,\bar{B} \in V_J$, we have $\|h_i - \bar{h}_i\|_{H^u} \lesssim \|B-\bar{B}\|_{H^{u+1}} \leq 2^{Ju} \|B-\bar{B}\|_{H^1}$ for $u\geq 0$ and $\|h_i - \bar{h}_i\|_\infty \lesssim \|B-\bar{B}\|_{B_{1\infty}^1} \lesssim 2^{Jd/2} \|B-\bar{B}\|_{H^1}$. Furthermore,
\begin{align*}
	\|h_i + \bar{h}_i\|_{H^u} \lesssim \sup_{B \in  \Scal\Bcal_T } 
	\|B-B_0\|_{H^{u+1}} & \lesssim \sup_{B \in  \Scal\Bcal_T } 
	\|B - P_J B_0\|_{H^{u+1}} + \|B_0 - P_J B_0\|_{H^{u+1}} \\
	& \lesssim 2^{Ju} \sup_{B \in  \Scal\Bcal_T } 
	\|B - P_J B_0\|_{H^1} + 2^{-J(s-u)} \|B_0\|_{H^{s+1}}.
\end{align*}
Note that $\|\nabla B - \nabla B_0\|_{L^2(\mu_0)} \simeq \|B - B_0\|_{H^1}$ since $\mu_0$ is both bounded and bounded away from zero and $B,B_0 \in \dot{L}^2$. Therefore, $\sup_{B \in  \Scal\Bcal_T } \|B - P_J B_0\|_{H^1} + 2^{-Js} \|B_0\|_{H^{s+1}} \leq C (M+1)\eps_T$, so that $\|h_i + \bar{h}_i\|_{H^u} \leq C (M+1) 2^{Ju} \eps_T$. Similarly, using the Sobolev embedding theorem,
\begin{align*}
	\|h_i + \bar{h}_i\|_\infty \lesssim \sup_{B \in  \Scal\Bcal_T } 
	\|B-B_0\|_{B_{1\infty}^1} 
	& \lesssim \sup_{B \in  \Scal\Bcal_T } \|B - P_J B_0\|_{B_{1\infty}^1} 
	+ \|B_0 - P_J B_0\|_{B_{1\infty}^1} \\
	& \lesssim 2^{Jd/2} \sup_{B \in  \Scal\Bcal_T } \|B - P_J B_0\|_{H^1} 
	+ 2^{-J(s-d/2)} \|B_0\|_{H^{s+1}} \\
	& \leq C(M+1) 2^{Jd/2} \eps_T.
\end{align*}
Combining these bounds yields
\begin{align*}
	d_L(f_B,f_{\bar{B}} ) & \leq C(M+1)2^{J[d/2+(d/2+\kappa-1)_+]} 
	\eps_T \| B -\bar{B}\|_{H^1}.
\end{align*}
Since $ \Scal\Bcal_T  \subseteq (V_J,\|\cdot\|_{H^1})$ is finite dimensional, using the last display and the usual covering argument for balls in finite dimensional spaces (e.g.~\cite{ginenickl2016}, Proposition 4.3.34),
\begin{align*}
	\log N(\mF_T,6d_L,\tau) 
	& \leq \log N( \Scal\Bcal_T , C(M+1)2^{J[d/2+(d/2+\kappa-1)_+]} 
	\eps_T \|\cdot\|_{H^1},\tau) \\
	& \lesssim \text{dim}(V_J) \log (C(M+1)2^{J[d/2+(d/2+\kappa-1)_+]} 
	\eps_T \sup_{B,\bar{B}\in  \Scal\Bcal_T } \|B-\bar{B}\|_{H^1} /\tau) \\
	& \lesssim \text{dim}(V_J) \log (CR_T /\tau),
\end{align*}
where $C>0$ and 
$$
	R_T := (M+1)^2 2^{J[d/2+(d/2+\kappa-1)_+]} \eps_T^2 \to 0
$$
under the current assumption on $s$ for $\kappa>0$ small enough. Recall the inequality $\int_0^a \sqrt{\log (A/x)}$ $dx \leq 4a\sqrt{\log (A/a)}$ for any $A \geq 2$ and $0<a\leq 1$ (\cite{ginenickl2016}, p. 190). Using this inequality, the last display with $\text{dim}(V_J) = O(2^{Jd})$ and that $\mF_T$ has $d_L$-diameter $D_{\mF_T} \lesssim R_T \to 0$,
\begin{align*}
	J_{\mF_T} 
	& \lesssim \text{dim}(V_J) \int_0^{D_{\mF_T}}
	\sqrt{\log ( [CR_T] \vee 2 /\tau) }d\tau  
	\lesssim 2^{Jd/2} D_{\mF_T} \sqrt{\log ( [CR_T] \vee 2 /D_{\mF_T}) }.
\end{align*}
Taking $D_{\mF_T} \simeq  R_T$ in the last display gives $J_{\mF_T}  \lesssim 2^{Jd/2} R_T  \left( 1 + \sqrt{\log (1/R_T)} \right).$ Arguing as for the bound for $d_L(f_B,f_{\bar{B}})$ above, one has for all $B \in  \Scal\Bcal_T $,
\begin{align*}
	\|L_{B_0}^{-1}[f_B]\|_\infty 
	\lesssim \|f_B\|_{H^{d/2+\kappa-1}} 
	& \lesssim \sum_{i=1}^d \Big{\|}h_i^2 - 
	\int_{\T^d} h_i^2 d\mu_0\Big{\|}_{H^{d/2+\kappa-1}} \\
	& \lesssim \sum_{i=1}^d \|h_i\|_\infty \|h_i\|_{H^{(d/2+\kappa-1)_+}}
	\lesssim 2^{J[d/2+(d/2+\kappa-1)_+]} (M+1)^2 \eps_T^2
\end{align*}
where the constants depend only on $d$, $\kappa$, $\|\mu_0\|_\infty$ and  $\|\nabla B_0\|_{B_{\infty\infty}^{|d/2+\kappa-1|}}\lesssim \|B_0\|_{H^{|d/2+\kappa-1|+d/2+1}}$ $\lesssim \|B_0\|_{H^{s+1}}$. Substituting this bound and $J_{\mF_T}  \lesssim 2^{Jd/2} R_T  \left( 1 + \sqrt{\log (1/R_T)} \right)$ into \eqref{chaining_inequality}, 
\begin{equation*}
	P_{B_0} \left( 
	\sup_{f_B \in \mF_T} | \G_T[f_B]| \geq C2^{Jd/2} 
	R_T \left(1 + \sqrt{\log (1/R_T)} \right) \left( 1 + x \right) \right) \leq e^{-x^2/2}
\end{equation*}
for any $x \geq 0$. Set 
$$
	\zeta_T 
	= M_T T^{-1/2} 2^{Jd/2} R_T (1 + \sqrt{\log (1/R_T)})
	= O(M_T\sqrt{\log T} T^{-\frac{s-d/2-(d/2+\kappa-1)_+}{2s+d}} \eps_T^2),
$$
which satisfies $\zeta_T = o(\eps_T^2)$ for $M_T \to \infty$ growing slow enough, since $s>d \vee (3/2) > d/2+(d/2+\kappa-1)_+$ for $\kappa>0$ small enough. Then using the definition of the empirical process $\G_T[f_B]$, for any $M_T \to \infty$, 
\begin{equation*}
	P_{B_0} \Bigg( \sup_{B \in  \Scal\Bcal_T } \Bigg|  \frac{1}{T} \int_0^T
	 \|\nabla B(X_s) - \nabla B_0(X_s)\|^2 ds - \|\nabla B - \nabla B_0\|
	 _{L^2(\mu_0)}^2 \Bigg| \geq \zeta_T  \Bigg) \to 0
\end{equation*}
as $T \to \infty$. The result then follows because on the complement of the event in the last display,
\begin{align*}
	\sup_{B \in  \Scal\Bcal_T } \frac{1}{T} \int_0^T  \|\nabla B(X_s)-
	\nabla B_0(X_s)\|^2 ds & \leq \sup_{B \in  \Scal\Bcal_T } 
	\|\nabla B - \nabla B_0\|_{L^2(\mu_0)}^2 + \zeta_T \\
	& \leq M^2 \eps_T^2 +o(\eps_T^2).
\end{align*}

\end{proof}

\subsection{Proof of Lemma \ref{lem:prior_prob_event}}

\begin{lemmaPriorProb}
	For $s,M,\kappa>0$, $p\in[1,2]$ and sequences  $\eps_T= T^{-s/(2s+d)},\ \overline\eps_T=T^{-(s+1)/(2s+d)}$, define the sets
$$
	\mB_T 
	=
		\{ B = B_1 + B_2: \|B_1\|_\infty \leq \overline{\eps}_T, 
		\|B_1\|_{C^1} \leq \eps_T,
		\|B_2\|_{B^{s+1}_{pp}} \leq M, 
		\|B\|_{C^{(d/2+\kappa)\vee 2}} \leq M \}.
$$
Assume either:
\begin{enumerate}
\item[(i)] $p=2$ and $B = W/(\sqrt{T}\eps_T)$ for $W\sim \Pi_W$ a Gaussian process satisfying Condition \ref{GP_condition};

\item[(ii)] $B = W/(T\eps_T^2)^\frac{1}{p}$ for $W\sim\Pi_W$ a $p$-exponential random element as in \eqref{Eq:pExpBasePrior} with $s>(d/2+\kappa)\vee 2+d/p-1$.
\end{enumerate}
Let $\Pi=\Pi_T$ be the law of $B$. Then, for every $K>0$, there exists $M>0$ large enough such that $\Pi(\mB_T^c) \leq e^{-KT\eps_T^2}.$

\end{lemmaPriorProb}

\begin{proof}

(i) In the Gaussian case, define the sets
\begin{align*}
	\mB_{T,1}&:=\{B = B_1 + B_2:  \|B_1\|_\infty \leq \overline{\eps}_T,
	 \|B_1\|_{C^1} \leq \eps_T,
	 \|B_2\|_{B^{s+1}_{22}}\leq M\},\\
	\mB_{T,2}&:=\{ B: \|B\|_{C^{(d/2+\kappa)\vee 2}} \leq M\}.
\end{align*}
To upper bound $\Pi(\mB_T^c)$ it thus suffices to upper bound $\Pi(\mB_{T,1}^c)$ and $\Pi(\mB_{T,2}^c)$.  Since $\|g\|_{B^{s+1}_{22}}=\|g\|_{H^{s+1}} \leq c_0\|g\|_\H$ for all $g\in\H$ under Condition \ref{GP_condition}, Borell's isoperimetric inequality (\cite{ginenickl2016}, Theorem 2.6.12) gives
\begin{align}
\label{Isop}
	\Pi(\mB_{T,1})
	&=
		\Pi_W(W = W_1 + W_2:  \|W_1\|_\infty \leq \sqrt{T}\eps_T
		\overline{\eps}_T,\|W_1\|_{C^1} \leq \sqrt{T}\eps_T^2,
		\|W_2\|_{B^{s+1}_{22}} \le M \sqrt{T}\eps_T) \nonumber\\
	& \geq 
		\Phi (\Phi^{-1}(\Pi_W(W: \|W\|_{\infty} \leq \sqrt{T}\eps_T\overline\eps_T, \|W\|_{C^1} \leq \sqrt{T}\eps_T^2)) +  M \sqrt{T}\eps_T/c_0),
\end{align}
where $\Phi$ is the standard normal cumulative distribution function. Now for $\H_1$ and $H_1^{s+1}$ the unit balls of $\H$ and $H^{s+1}$ respectively, we have under Condition \ref{GP_condition} that
\begin{align*}
	\log N(\H_1,\|\cdot\|_{C^1},\tau) \leq \log N(  c_0H^{s+1}_1,\|\cdot\|_{C^1},\tau)
	\lesssim \tau^{-d/s},
\end{align*}
where the last inequality follows by arguing as in Theorem 4.3.36 of \cite{ginenickl2016}.
By Theorem 1.2 of Li and Linde \cite{li1999}, this yields
\begin{equation}\label{SB_C1}
\Pi_W(\|W\|_{C^1} \leq \eta) \geq e^{-c_1^2 \eta^{-2d/(2s-d)}} \qquad \text{as } \eta \to 0,
\end{equation}
for any $d/2<s<\infty$ and some $c_1 = c_1(d,s,c_0)>0$, which implies, since $\sqrt{T}\eps_T^2\to0$,
$$
	\Pi_W(\|W\|_{C^1} \leq \sqrt{T}\eps_T^2 ) 
	\geq e^{-c_1^2 (\sqrt{T}\eps_T^2)^{-2d/(2s-d)}}
	= e^{-c_1^2 T\eps_T^2}.
$$
Using the same argument, now with the bound
\begin{align*}
	\log N(\H_1,\|\cdot\|_\infty,\tau) \leq 
	\log N(  c_0H^{s+1}_1,\|\cdot\|_\infty,\tau) \lesssim \tau^{-d/(s+1)},
\end{align*}
it follows for some $c_2>0$ that $\Pi_W(\|W\|_{\infty} \leq \sqrt{T}\eps_T\overline\eps_T) \ge e^{-c_2^2 T\eps_T^2}.$
The Gaussian correlation inequality (which holds for Gaussian measures in separable Banach spaces, see 
Lemma \ref{lem:InfiniteIneq} below) then gives for $c_3=c_1^2+c_2^2$
\begin{align*}
	\Pi_W(\|W\|_{\infty} \leq \sqrt{T}\eps_T\overline\eps_T, 
	\|W\|_{C^1} \leq \sqrt{T}\eps_T^2)
&\ge
	\Pi_W( \|W\|_{\infty} \leq \sqrt{T}\eps_T\overline\eps_T)
	\Pi_W( \|W\|_{C^1} \leq \sqrt{T}\eps_T^2)\\
&\ge
	e^{-c_3T\eps_T^2}.
\end{align*}
Using the standard inequality $\Phi^{-1}(y) \geq -\sqrt{2\log (1/y)}$ for $0<y<1$, the right hand side of \eqref{Isop} is thus lower bounded by
$$
	\Phi( ( M /c_0- \sqrt{2c_3}) \sqrt{T}\eps_T ).
$$
Defining $m_T = -\Phi^{-1}(e^{-KT\eps_T^2}/2)$, this further gives $m_T \leq \sqrt{2\log 2} + \sqrt{2KT}\eps_T$, which can be made smaller than $( M /c_0- \sqrt{2c_3})\sqrt{T}\eps_T$ by taking $M = M(K,c_0,c_3)$ large enough. For such $M$, the last display is lower bounded by $\Phi(m_T) = 1-\Phi (\Phi^{-1}(e^{-KT\eps_T^2}/2)) = 1-\tfrac{1}{2}e^{-KT\eps_T^2}$.

	To bound  $\Pi(\mB_{T,2}^c)$, recall that by Condition \ref{GP_condition}, $W$ defines a Gaussian Borel random element in a separable linear subspace $\mathcal{S}$ of $C^{(d/2+\kappa)\vee 2}$. Using the Hahn-Banach theorem, we may thus represent its norm as
$$
	\|W\|_{C^{(d/2+\kappa)\vee 2}} = \sup_{L \in \mathcal{L}} |L(W)|,
$$
where $\mathcal{L}$ is a countable set of bounded linear functionals on $(\mathcal{S},\|\cdot\|_{C^{(d/2+\kappa)\vee 2}})$. Applying Fernique's theorem \cite[Theorem 2.1.20]{ginenickl2016} to the centred Gaussian process $(X(L) = L(W): L \in \mathcal{L})$, we have $E\|W\|_{C^{(d/2+\kappa)\vee 2}} = E \sup_{L\in\mathcal{L}}|X(L)| \leq D < \infty$, and for $ M =M(D)>0$ large enough and since $\sqrt{T}\eps_T \to \infty$,
$$
	\Pi(\mB_{T,2}^c) 
	\leq \Pi_W \left(W: \|W\|_{C^{(d/2+\kappa)\vee 2}} 
	- E\|W\|_{C^{(d/2+\kappa)\vee 2}} 
	\geq M\sqrt{T}\eps_T \right) \leq 2 e^{-D'M^2 T\eps_T^2}
$$
for some fixed constant $D'>0$. Taking $M>0$ large enough, the right-hand side can be made less than $\tfrac{1}{2}e^{-KT\eps_T^2}$, concluding the proof.

\smallskip

	(ii). Turning to $p$-exponential priors, define the set
\begin{align*}
	\Bcal'_T
	=
		\big\{B=B'_1+B'_2+B_3'\ :\ & B'_i\in V_J\cap \dot L^2, \ \|B'_{1}\|_{\infty}
		\le \overline \eps_T,\\
	&
		\|B'_{2}\|_{H^{s+1+d/2-d/p}}\le  M^\frac{p}{2}
		(T\eps_T^2)^{\frac{1}{2}-\frac{1}{p}},\
		\ \|B'_3\|_{B^{s+1}_{pp}}\le M \big\}.
\end{align*}
We lower bound the prior probability of $\Bcal'_T$ using the generalization of Borell’s inequality to $p$-exponential measures. The space of admissible shifts (cfr. Proposition 2.7 in \cite{ADH20}) of the scaled $p$-exponential random element $B=W/(T\eps_T^2)^\frac{1}{p}$ is $\Qcal = V_J\cap \dot{L}^2$, with norm
$$
	\|h\|_\Qcal
	=
		(T\eps_T^2)^\frac{1}{p}\Bigg(\sum_{l =0}^J\sum_r
		2^{2 l \big(s+1+\frac{d}{2}-\frac{d}{p}\big)}|\langle h,
		\Phi_{l r}\rangle_2|^2\Bigg)^\frac{1}{2}
	=
		\big(T\eps_T^2\big)^\frac{1}{p}\|h\|_{H^{s+1+d/2-d/p}},
	\quad h\in\Qcal.
$$
Then, recalling the $\Zcal$-norm defined in \eqref{Eq:Znorm}, Proposition 2.15 in \cite{ADH20} implies
\begin{align*}
	\Pi(\Bcal'_T)
	&=
		\Pi\big(B=B'_{1}+B'_{2}+B'_{3}: \|B'_{1}\|_{\infty}\le \overline\eps_T,
		\|B'_{2}\|_{\Qcal}\le (M^pT\eps_T^2)^\frac{1}{2},
		\|B'_{3}\|_{\Zcal}\le M(T\eps_T^2)^\frac{1}{p} \big)\\
	&\ge
		1-\frac{1}{\Pi(\|B\|_\infty
		\le\overline\eps_T\big)} \exp \left( -(M^p/k)T\eps_T^2 \right)
\end{align*}
for some $k=k(p)>0$. By an analogous small ball computation as in the proof of Proposition 6.3 in \cite{ADH20}, it follows that as $T\to\infty$
\begin{align*}
	-\log \Pi(\|B\|_\infty\le \overline{\eps}_T)
	&=
		-\log \Pi_W (\|W\|_\infty \le \overline{\eps}_T (T\eps_T^2)^\frac{1}{p} )\\
	&\simeq 
		[\overline{\eps}_T (T\eps_T^2)^\frac{1}{p} ]^\frac{d}{s+1-d/p}
	= T\eps_T^2.
\end{align*}
Thus, for some constant $c_1=c_1(s,p,d)>0$, we have $\Pi(\Bcal_T') \geq 1-e^{-[(M^p/k)-c_1]T\eps_T^2},$ so that for any $K>0$ we can choose $M=M(K,c_1,k)=M(K,s,p,d)>0$ large enough to obtain $\Pi(\Bcal'_T) \ge 1-e^{-KT\eps_T^2}.$

	We conclude the proof by showing that $\Bcal_T'\subset \Bcal_T$. First, since $B'_1\in V_J$ we have $\|B'_1\|_{C^1}\le 2^J \| B'_1 \|_{\infty} \lesssim T^{1/(2s+d)} \overline\eps_T = \eps_T$, 
and therefore both norm bounds on $B_1=B_1'$ in \eqref{B_T set} are satisfied.
Also note that 
\begin{equation}
\label{Eq:AnotherBound}
	\|B'_1\|_{C^{(d/2+\kappa)\vee2}}\le 2^{J(d/2+\kappa)\vee2} \| B'_1 \|_{\infty}
	\lesssim T^{\frac{(d/2+\kappa)\vee2}{2s+d}} \overline\eps_T = o(1)
\end{equation}
since by assumption $s+1>(d/2+\kappa)\vee 2+d/p$. Next, since also $B_2' \in V_J$, using the continuous embedding $H^{s+1}(\T^d)\subseteq B^{s+1}_{pp}(\T^d), \ p\le2$ (p.33 in \cite{LSS09}),
\begin{align*}
	\|B'_{2}\|_{B_{pp}^{s+1}}
	\lesssim
		\|B'_{2}\|_{H^{s+1}}
	& \le
		2^{Jd\left(\frac{1}{p}-\frac{1}{2}\right)}\|B_{2}'\|_{H^{s+1+d/2-d/p}} \\
	&\leq 
		2^{Jd\left(\frac{1}{p}-\frac{1}{2}\right)} M^\frac{p}{2}
		(T\eps_T^2)^{\frac{1}{2}-\frac{1}{p}}  
	\simeq 
		M^\frac{p}{2}.
\end{align*}
Thus, taking $B_2=B_2'+B_3'$ implies $\|B_2\|_{B^{s+1}_{pp}}\lesssim M^{p/2}+M \lesssim M$ for $M\geq 1$ as required. Finally, using \eqref{Eq:AnotherBound} and the continuous embedding $B^{s+1}_{pp} \subset C^{(d/2+\kappa)\vee 2}$ holding for all $s+1> (d/2+\kappa)\vee 2+d/p$ (p.170 in \cite{ST87}, using the embedding $B^{s+1}_{pp}\subset B^{s'}_{p1}\subset C^{(d/2+\kappa)\vee 2}$ holding for $s+1> s'> (d/2+\kappa)\vee 2+d/p$) if $d\le 3$, in which in case $(d/2+\kappa)\vee 2 = 2\in\N$ for small $\kappa$),
$$
	\|B\|_{C^{(d/2+\kappa)\vee 2}}
	\le \|B'_1\|_{C^{(d/2+\kappa)\vee 2}}
	+ \|B_2\|_{B^{s+1}_{pp}} \lesssim M,
$$
concluding the proof.
\end{proof}

\subsection{Proof of Lemma \ref{lem:bias}}

\begin{lemmaBias}
For $s,M,\kappa>0$ and $p\in[1,2]$, let $\mB_T$ be the set in \eqref{B_T set}, with $\eps_T, \overline \eps_T$ as in Lemma \ref{lem:prior_prob_event}. If $2^J \simeq T^{1/(2s+d)}$, then there exists a finite constant $C$ depending on $s,d,m$ and the wavelet basis $\{\Phi_{l r}\}$ such that
$$ 
	\left\{ \mu_B = \frac{e^{2B}}{\int_{\T^d} e^{2B(x)}dx}: B \in \mB_T \right\} 
	\subset \{ \mu: \|\mu - P_J\mu\|_{W^{1,p}} \leq C\eps_T\}.
$$

\end{lemmaBias}

\begin{proof}

Since $\|B\|_\infty \leq \|B\|_{C^2} \leq M$ for every $B\in \mB_T$, this implies $e^{-2M} \leq \int_{\T^d} e^{2B}dx \leq e^{2M}$ and hence $\|\mu_B - P_J\mu_B\|_{W^{1,p}} \leq e^{2M} \|e^{2B} - P_Je^{2B}\|_{W^{1,p}}$, so it suffices to bound the last quantity. For a function $f$ on the torus $\T^d$, denote by $\bar{f}$ its periodic extension to $\R^d$. Recall that the periodic projection satisfies $P_Jf(x) = \int_{\R^d} K_J(x,y) \bar{f}(y) dy$ for all $x\in (0,1]^d$, where $K_J(x,y) = 2^{Jd} \sum_{k \in \Z^d} \phi(2^J x-k) \phi(2^J y- k)$ is the unperiodized wavelet kernel and $\phi$ is the unperiodized father wavelet used in the construction of the periodized wavelet basis (see (4.127) in \cite{ginenickl2016}). Using that $\int_{\R^d} K_J(x,y)dy = 1$ for all $x\in (0,1]^d$ and writing $B=B_1 + B_2$ as in \eqref{B_T set},
\begin{align*}
	|\partial_{x_i}(e^{2B}-P_Je^{2B})(x)|
	&=
		\Big| \partial_{x_i}\int_{\R^d}K_J(x,y)
		(e^{2B_1(x)+2B_2(x)}-e^{2\bar B_1(y)+2\bar B_2(y)}) 
		dy\Big|\\
	&\le
		\Big| \partial_{x_i}\Big(e^{2B_1(x)}\int_{\R^d}K_J(x,y)
		(e^{2B_2(x)}-e^{2\bar B_2(y)}) dy\Big)\Big|\\
	& \quad +
		\Big| \partial_{x_i}\Big(\int_{\R^d}K_J(x,y)(e^{2B_1(x)}
		-e^{2\bar B_1(y)})e^{2\bar B_2(y)}dy\Big)\Big| \\
	& \leq e^{2B_1(x)} \Big|2\partial_{x_i}B_1(x)
	[e^{2B_2(x)}-P_Je^{2B_2}(x)]
	+\partial_{x_i}[e^{2B_2(x)}-P_Je^{2B_2}(x)]
	\Big|\\
	& \quad +
		\Big|\int_{\R^d}\partial_{x_i}K_J(x,y)(e^{2B_1(x)}
		-e^{2\bar B_1(y)})e^{2\bar B_2(y)}dy\Big|\\
	& \quad +
		e^{2B_1(x)}\Big|2\partial_{x_i}B_1(x)\int_{\R^d}K_J(x,y)
		e^{2\bar B_2(y)}dy\Big|.
\end{align*}
Using that $\int_{\R^d}\partial_{x_i}K_J(x,y)dy\lesssim 2^J$ by the localization property of wavelets and that $|e^x-1| \lesssim x$ for small $|x|$, the last display is bounded by a multiple of
\begin{align*}
	e^{2\|B_1\|_\infty}& \|B_1\|_{C_1}
	\Big|(e^{2B_2}-P_Je^{2B_2})(x)\Big|
	+e^{2\|B_1\|_\infty} \Big|\partial_{x_i}(e^{2B_2}-P_Je^{2B_2})(x)\Big|\\
	&\quad
		+2^Je^{2\|B_2\|_\infty}\|B_1\|_\infty
		+e^{2\|B_1\|_\infty}\|B_1\|_{C^1}|P_Je^{2 B_2}(x)|.
\end{align*}
Taking $p^{th}$ powers and integrating, and using the embedding $B^0_{p1}(\T^d) \subset L^p(\T^d)$ (Theorem 1, p.163 in \cite{ST87}), then yields
\begin{align*}
	\|\partial_{x_i}(e^{2B}-P_Je^{2B})\|_{p}
	&\lesssim
		e^{2\|B_1\|_\infty}\|B_1\|_{C_1}
		\|e^{2 B_2}-P_Je^{2 B_2}\|_{B^0_{p1}}
		+e^{2\|B_1\|_\infty}\|e^{2 B_2}-P_Je^{2 B_2}\|_{B^1_{p1}}\\
	&\quad
		+2^Je^{2\|B_2\|_\infty}\|B_1\|_\infty
		+e^{2\|B_1\|_\infty}\|B_1\|_{C_1}
		\|P_Je^{2 B_2}\|_{B^0_{p1}}.
\end{align*}
By Lemma \ref{lem:exp_map}, $\|e^{2 B_2}\|_{B^{s+1}_{pp}} \lesssim 1+\|B_2\|_{B^{s+1}_{pp}} + \|B_2\|_{B^{s+1}_{pp}}^{s+1}\leq c(m,s,p)$, which implies
\begin{align*}
	\|e^{2 B_2}-P_Je^{2 B_2}\|_{B^1_{p1}}
	&=
		\sum_{ l >J} 2^{l  (1+d/2-d/p)}\Big(
		\sum_{r}|\langle e^{2 B_2},\Phi_{l r}\rangle_2 |^p\Big)^{1/p}\\
	&=
		\sum_{ l >J} 2^{- l s}\Big( 2^{p  l (s+1+d/2-d/p)}
		\sum_{r}|\langle e^{2 B_2},\Phi_{l r}\rangle_2 |^p\Big)^{1/p}\\
	&\le
		\|e^{2B_2}\|_{B^{s+1}_{pp}}\sum_{ l >J} 2^{- l s}
	\lesssim
		2^{-J s}.
\end{align*}
By a similar computation, $\|e^{2 B_2}-P_Je^{2 B_2}\|_{B^0_{p1}}\lesssim 2^{-J(s+1)}$, while by H\"older's inequality with exponents $(p/(p-1),p)$,
\begin{align*}
	\|P_J e^{2 B_2}\|_{B^0_{p1}}
	&=
		\sum_{ l \le J} 2^{l  (d/2-d/p)}\Big(\sum_r
		|\langle e^{2 B_2},\Phi_{l r}\rangle_2 |^p\Big)^{1/p}\\
	&\le
		\Big( \sum_{ l \le J}[2^{- l (s+1)} ]^\frac{p}
		{p-1}\Big)^\frac{p-1}{p}
		\Big(\sum_{l \le J} 2^{p l (s+1+d/2-d/p)}\sum_r
		|\langle e^{2 B_2},\Phi_{l r}\rangle_2 |^p\Big)^{1/p}\\
	&\lesssim
		\|e^{2B_2}\|_{B^{s+1}_{pp}}
	\lesssim 
		1.
\end{align*}
Combining the above bounds, using the definition of $\eps_T, \overline\eps_T$ in Lemma \ref{lem:prior_prob_event} and that $2^J\simeq T^{1/(2s+d)}$, we thus obtain that for any $B \in \mathcal{B}_T$, 
\begin{align*}
	\|\partial_{x_i}(e^{2B}-P_Je^{2B})\|_{p}
	& \lesssim
		e^{2\overline{\eps}_T} \eps_T 2^{-J(s+1)}
		+e^{2\overline{\eps}_T} 2^{-Js}
		+2^J e^{2M}\overline{\eps}_T
		+e^{2\overline{\eps}_T} \eps_T
		\lesssim  \eps_T.
\end{align*}
By a similar, in fact easier, computation, we also obtain
$$
	\|e^{2B}-P_Je^{2B}\|_{p} \lesssim e^{2\bar \eps_T} 2^{-J(s+1)}+ e^{2M}
	\bar \eps_T = o(\eps_T).
$$
The required bias bound then follows from the last two displays.
\end{proof}

\end{appendix}

\end{supplement}


\bibliographystyle{imsart-number} 
\bibliography{rev_diffusion_bvm}       

\end{document}